%% file: edd.tex
\RequirePackage{latexsym, amsmath, amstext, amssymb, amsfonts, amscd, bm, array, amsbsy, mathrsfs, amscd}
\documentclass[11pt, cmp, numbook, final, a4paper]{svjour}
\usepackage[11pt]{extsizes}
\pdfoutput=1
\usepackage{indentfirst}
\usepackage{graphics} 
\usepackage{epsfig}
\usepackage{transparent, colortbl}
\usepackage{booktabs}
\usepackage{multirow}
\usepackage{anysize}
\usepackage{latexsym}
\usepackage{pdfsync}
\usepackage[all]{xy}
\usepackage[usenames,dvipsnames]{xcolor}
\usepackage{enumerate}
\usepackage{float}
\restylefloat{table}
\usepackage{upgreek}
\usepackage{array}
\usepackage[colorlinks=true,linkcolor=black]{hyperref}
\colorlet{linkequation}{blue}
\usepackage{multicol}

\newcommand*{\SavedEqref}{}
\let\SavedEqref\eqref
\renewcommand*{\eqref}[1]{%
  \begingroup
    \hypersetup{
      linkcolor=blue,
      linkbordercolor=blue,
    }%
    \SavedEqref{#1}%
  \endgroup
}

\def\bk{\mathbf{k}}
\def\kdim{\mathrm{kdim}}
\def\MCM{\mathrm{MCM}}

\def\opp{\mathrm{opp}}

\DeclareSymbolFont{extraup}{U}{zavm}{m}{n}
\DeclareMathSymbol{\varheart}{\mathalpha}{extraup}{86}
\DeclareMathSymbol{\vardiamond}{\mathalpha}{extraup}{87}

\makeatletter
\def\CT@@do@color{%
  \global\let\CT@do@color\relax
        \@tempdima\wd\z@
        \advance\@tempdima\@tempdimb
        \advance\@tempdima\@tempdimc
        \kern-\@tempdimb
\transparent{0.6}%
        \leaders\vrule
                \hskip\@tempdima\@plus  1fill
        \kern-\@tempdimc
        \hskip-\wd\z@ \@plus -1fill }
\makeatother

\makeatletter
\newcommand{\thickhline}{%
    \noalign {\ifnum 0=`}\fi \hrule height 1pt
    \futurelet \reserved@a \@xhline
}
\newcolumntype{"}{@{\hskip\tabcolsep\vrule width 1pt\hskip\tabcolsep}}
\makeatother

\newtheorem{Theorem}{Theorem}[section]
\newtheorem{Lemma}[Theorem]{Lemma}
\newtheorem{Proposition}[Theorem]{Proposition}
\newtheorem{Corollary}[Theorem]{Corollary}

\newtheorem{Definition}[Theorem]{Definition}

\newcommand{\Hom}{{\rm Hom}}
\newcommand{\A}{{\Bbb{A}}}

\newcommand{\Spec}{\mathrm{Spec}}

\newcommand{\im}{\mathrm{im}}

\newcommand{\Ext}{\mathrm{Ext}}

\newcommand{\bp}{\begin{Proposition}}
\newcommand{\ep}{\end{Proposition}}
\newcommand{\bl}{\begin{Lemma}}
\newcommand{\el}{\end{Lemma}}
\newcommand{\bt}{\begin{Theorem}}
\newcommand{\et}{\end{Theorem}}
\newcommand{\bd}{\begin{Definition}}
\newcommand{\ed}{\end{Definition}}
\newcommand{\End}{\mathrm{End}}

\newcommand{\Mod}{\mathrm{Mod}}

\newcommand{\Mat}{\mathrm{Mat}}

\newcommand{\eqdef}{\stackrel{{\rm def.}}{=}}

\DeclareFontFamily{U}{rsf}{}
\DeclareFontShape{U}{rsf}{m}{n}{<5> <6> rsfs5 <7> <8> <9> rsfs7 <10-> rsfs10}{}
\DeclareMathAlphabet\Scr{U}{rsf}{m}{n}

\def\N{\mathbb{N}}
\def\Z{\mathbb{Z}}
\def\C{\mathbb{C}}

\def\rk{{\rm rk}}

\def\GL{\mathrm{GL}}

\def\supp{\mathrm{supp}}

\def\fd{\mathfrak{d}}

\def\rD{\mathrm{D}}

\def\sing{\mathrm{sing}}

\def\cC{\mathcal{C}}

\def\MF{\mathrm{MF}}
\def\HMF{\mathrm{HMF}}

\def\fp{\mathrm{fp}}
\def\MF{\mathrm{MF}}
\def\HMF{\mathrm{HMF}}
\def\ZMF{\mathrm{ZMF}}
\def\BMF{\mathrm{BMF}}
\def\EF{\mathrm{EF}}
\def\HEF{\mathrm{HEF}}
\def\ZEF{\mathrm{ZEF}}

\def\Z{\mathbb{Z}}

\def\Perf{\mathrm{Perf}}
\def\bm{\mathbf{m}}

\def\sing{\mathrm{sing}}
\def\loc{\mathrm{loc}}
\def\uloc{\underline{\loc}}

\def\depth{\mathrm{depth}}
\def\mhef{\mathbf{hef}}

\newcommand{\be}{\begin{equation*}}
\newcommand{\ee}{\end{equation*}}
\newcommand{\ben}{\begin{equation}}
\newcommand{\een}{\end{equation}}
\newcommand{\beqa}{\begin{eqnarray*}}
\newcommand{\eeqa}{\end{eqnarray*}}
\newcommand{\beqan}{\begin{eqnarray}}
\newcommand{\eeqan}{\end{eqnarray}}

\newcommand \QQ {{\mathbb Q}}

\def\2{\bar{2}}
\def\3{\bar{3}}
\def\4{\bar{4}}

\newcommand{\id}{\mathrm{id}}
\newcommand{\Tr}{\mathrm{Tr}}

\newcommand{\diag}{\mathrm{diag}}

\def\lto{\longrightarrow}

\def\O{\mathrm{O}}

\def\cA{\mathcal{A}}

\def\cP{\mathcal{P}}

\def\cT{\mathcal{T}}

\def\G_2{\mathrm{G_2}}

\def\fI{\mathfrak{I}}

\def\Ann{\mathrm{Ann}}

\def\mf{\mathbf{f}}

\def\f{\mathfrak{f}}

\def\Ob{\mathrm{Ob}}

\def\sc{\mathrm{sc}}

\def\HF{\mathrm{HF}}

\def\0{{\hat{0}}}
\def\1{{\hat{1}}}

\def\O{\mathrm{O}}

\def\mod{\mathrm{mod}}

\def\rD{\mathrm{D}}

\def\uHom{\underline{\Hom}}

\def\Cok{\mathrm{Cok}}
\def\cok{\mathrm{cok}}

\def\MCM{\mathrm{MCM}}
\def\uMCM{\underline{\MCM}}
\def\umod{\underline{\mod}}
\def\hmf{\mathrm{hmf}}
\def\zmf{\mathrm{zmf}}
\def\bmf{\mathrm{bmf}}
\def\mf{\mathrm{mf}}
\def\zef{\mathrm{zef}}
\def\hef{\mathrm{hef}}

\def\rad{\mathrm{rad}}
\def\Irr{\mathrm{Irr}}

\def\bd{\mathbf{d}}
\def\bdelta{\boldsymbol{\delta}}

\newcommand{\twopartdef}[4]
{
	\left\{
		\begin{array}{ll}
			#1 & \mbox{if } #2 \\
			#3 & \mbox{if } #4
		\end{array}
	\right.
}

\newcommand{\fourpartdef}[8]
{
	\left\{
		\begin{array}{ll}
			#1 & \mbox{if } #2 \\
			#3 & \mbox{if } #4 \\
                        #5 & \mbox{if } #6 \\
                        #7 & \mbox{if } #8
		\end{array}
	\right.
}

\def\depth{\mathrm{depth}}

\begin{document}

\title{Matrix factorizations over elementary divisor domains}

\author{Dmitry Doryn$^1$, Calin Iuliu Lazaroiu$^1$, Mehdi Tavakol$^2$}

\institute{$^1$Center for Geometry and Physics, Institute for Basic
  Science, Pohang 37673, Republic of Korea\\
$^2$Max-Planck Institut f\"ur Mathematik, Vivatsgasse 7, 
 Bonn 53111, Germany  \\
  \email{dmitry@ibs.re.kr, calin@ibs.re.kr, mehdi@mpim-bonn.mpg.de
}}

\date{}
 
\maketitle

\abstract{We study the homotopy category $\hmf(R,W)$ of matrix
factorizations of non-zero elements $W\in R^\times$, where $R$ is an
elementary divisor domain. When $R$ has prime elements and $W$ factors
into a square-free element $W_0$ and a finite product of primes of
multiplicity greater than one and which do not divide $W_0$, we show
that $\hmf(R,W)$ is triangle-equivalent with an orthogonal sum of the
triangulated categories of singularities $\rD_\sing(A_n(p))$ of the
local Artinian rings $A_n(p)=R/\langle p^n\rangle$, where $p$ runs
over the prime divisors of $W$ of order $n\geq 2$. This result holds
even when $R$ is not Noetherian. The triangulated categories
$\rD_\sing(A_n(p))$ are Krull-Schmidt and we describe them
explicitly. We also study the cocycle category $\zmf(R,W)$, showing
that it is additively generated by elementary matrix
factorizations. Finally, we discuss a few classes of examples.}

\tableofcontents

\section*{Introduction}

The study of open-closed topological Landau-Ginzburg models
\cite{tft,LG1,LG2,lg1} defined on a Stein manifold $X$ \cite{lg2}
leads naturally to the problem of understanding categories of
finitely-generated projective factorizations over the non-Noetherian
ring $\O(X)$ of holomorphic complex-valued functions defined on
$X$. The simplest interesting models of this type arise when $X$ is an
arbitrary borderless, smooth and connected non-compact Riemann surface
$\Sigma$ (which may have infinite genus), with superpotential given by
a non-vanishing holomorphic function $W:\Sigma\rightarrow \C$. In this
situation, the ring $R=\O(X)$ is a so-called elementary divisor domain
(see Appendix \ref{app:edd}), i.e. it has the property that any matrix
with entries from $R$ admits a Smith normal form. Since any elementary
divisor domain is a B\'ezout domain, this implies that any
finitely-generated projective $R$-module is free (see \cite{FS}),
hence the relevant category reduces to the usual homotopy category
$\hmf(R,W)$ of finite rank matrix factorizations over $R$.

In the present paper, we consider a similar problem for {\em any}
elementary divisor domain $R$ which has prime elements, showing that
the triangulated structure of the category $\hmf(R,W)$ can be
determined explicitly for any element $W\in R^\times$ which is {\em
critically-finite}, i.e. which can be written as a product $W=W_0W_c$,
where the {\em non-critical part} $W_0$ is a square-free element of
$R^\times$ and the {\em critical part} $W_c$ is a (non-empty) finite
product of prime elements of $R$, each of which has multiplicity
strictly greater than one (we also require that $W_0$ and $W_c$ are coprime). 
More precisely, we will prove the following
result, which can be viewed as a non-Noetherian extension of the
Buchweitz correspondence \cite{Buchweitz} to elementary divisor
domains:

\

\begin{Theorem}
\label{thm:main} 
Let $R$ be an elementary divisor domain which has prime elements and
$W$ be a critically-finite element of $R$ with critical part
$W_c=p_1^{n_1}\ldots p_N^{n_N}$, where $p_1,\ldots, p_N$ (with $N\geq
1$) are prime elements of $R$ which are not mutually associated in
divisibility and $n_i\geq 2$. Then there exist equivalences of
triangulated categories:
\ben
\label{tequiv} 
\hmf(R,W)\simeq \vee_{i=1}^N
\umod_{R/\langle
p^{n_i}\rangle}\simeq \vee_{i=1}^N \rD_\sing(R/\langle
p^{n_i}\rangle)~~,
\een
where $\umod_{R/\langle p^{n_i}\rangle}\simeq \rD_\sing(R/\langle
p^{n_i}\rangle)$ denotes the projectively-stabilized category of
finitely-generated modules (a.k.a. the category of singularities) of
the ring $R/\langle p^{n_i}\rangle$, a ring which is Artinian.
\end{Theorem}

\

\noindent Our proof relies on the fact that matrices over an
elementary divisor domain admit a Smith normal form, which allows us
to reduce the problem to understanding certain properties of
elementary matrix factorizations (i.e. those matrix factorizations
whose reduced rank equals one).  The latter were studied in
\cite{elbezout} for any B\'ezout domain. The triangulated categories
$\rD_\sing(R/\langle p^{n_i}\rangle)$ are Krull-Schmidt and they admit
Auslander-Reiten triangles; their Auslander-Reiten quivers are
determined in Section \ref{sec:artinian}. Together with Theorem
\ref{thm:main}, this gives a complete description of the category
$\hmf(R,W)$ when the hypothesis of the theorem is satisfied.

The paper is organized as follows. In Section \ref{sec:bezout}, we
recall a few definitions and constructions for matrix factorizations
over B\'ezout domains. In Section \ref{sec:artinian}, we discuss
finitely-generated modules over the quotient of a B\'ezout domain by a
principal primary ideal. Section \ref{sec:edd} considers the homotopy
category of matrix factorizations over an elementary divisor domain
for a critically-finite $W$, giving the proof of Theorem
\ref{thm:main}. Section 4 discusses some examples, while the
appendices collect information about matrices over greatest common
divisor (GCD) domains and about elementary divisor domains (EDD).

\paragraph{Notations and conventions}
We use the same notations and conventions as in \cite{elbezout}.  In
particular, given an element $x$ of a unital commutative ring $R$, the
symbol $(x)\in R/U(R)$ (where $U(R)$ is the group of units of $R$)
denotes the class of $x$ under association in divisibility. When $R$
is a GCD domain (see Appendix \ref{app:GCD}) and $x_1,\ldots, x_n\in
R$, the symbol $(x_1,\ldots, x_n)\in R/U(R)$ denotes the association
in divisibility class formed by the greatest common divisors of
$x_1,\ldots, x_n$. The symbol $\langle x_1,\ldots, x_n\rangle$ denotes
the ideal generated by $x_1,\ldots, x_n$. The symbol $\Z_2$ stands for
the field $\Z/2\Z$, whose elements we denote by $\0$ and $\1$. The
symbol $\N$ denotes the set of natural numbers {\em including zero},
while $\N^\ast\eqdef \N\setminus \{0\}$.

\section{Matrix factorizations over a B\'ezout domain} 
\label{sec:bezout} 

\noindent Categories of matrix factorizations over a B\'ezout domain
were studied in \cite{elbezout}, to which we refer the reader for more
detail. In this section, we recall some definitions and
constructions which will be used later on. Let $R$ be a B\'ezout domain
and $W\in R^\times$ be a non-zero element of $R$.

\subsection{Categories of matrix factorizations over $R$}

\noindent As in \cite{elbezout}, we consider the following categories:
\begin{itemize} 
\itemsep 0.0em
\item The $R$-linear and $\Z_2$-graded differential category
$\MF(R,W)$ of finite rank matrix factorizations of $W$ over $R$. Its
objects are pairs $a=(M,D)$ with $M$ a free $\Z_2$-graded
$R$-module of finite rank and $D$ an odd endomorphism of $M$ such
that $D^2=W\id_M$. Since $W$ is non-vanishing, the even and odd
components of $M$ have equal rank, which we denote by $\rho(a)$ and
call the {\em reduced rank} of $a$; we have $\rk M=2\rho(a)$.
Choosing a $\Z_2$-homogeneous basis of $M$ allows us to identify $M$
with the $R$-supermodule $R^{\rho(a)|\rho(a)}$ whose
$\Z_2$-homogeneous components are both equal to the free module
$R^{\oplus \rho(a)}$.  Then $D$ identifies with a square matrix of
size $2\rho(a)$ in block off-diagonal form:
\be
D=\left[\begin{array}{cc} 0 & v\\ u & 0\end{array}\right]~~,
\ee
where $u$ and $v$ are square matrices of size $\rho(a)$ with entries in
$R$. The condition $D^2=W\id_M$ amounts to the relations: 
\ben
\label{uvrels}
uv=vu=W I_{\rho(a)}~~,
\een
where $I_{\rho(a)}$ is the identity matrix of size $\rho(a)$. Since
$W\neq 0$, these conditions imply that the matrices $u$ and $v$ have
maximal rank. Given two objects $a_1=(M_1,D_1)$ and $a_2=(M_2,D_2)$
of $\MF(R,W)$, the $\Z_2$-graded $R$-module of morphisms from $a_1$ to
$a_2$ is given by the inner $\Hom$: 
\be
\Hom_{\MF(R,W)}(a_1,a_2)=\uHom_R(M_1,M_2)=\Hom_R^\0(M_1,M_2)\oplus
\Hom_R^\1(M_1,M_2)~~, 
\ee 
endowed with the differential determined by the condition: 
\be 
\fd_{a_1,a_2}(f)=D_2\circ f-(-1)^\kappa
f\circ D_1~~,~~\forall f\in \Hom_R^\kappa(M_1,M_2)~~, 
\ee where
$\kappa\in \Z_2$.
\item The $R$-linear and $\Z_2$-graded cocycle, coboundary and total
cohomology categories $\ZMF(R,W)$, $\BMF(R,W)$ and $\HMF(R,W)$ of
$\MF(R,W)$.
\item The subcategories $\mf(R,W)$, $\zmf(R,W)$, $\bmf(R,W)$ and
$\hmf(R,W)$ obtained from $\MF(R,W)$, $\ZMF(R,W)$, $\BMF(R,W)$ and
$\HMF(R,W)$ by restricting to morphisms of even degree. Notice that
$\hmf(R,W)$ is the usual homotopy category of finite rank matrix factorizations.
\end{itemize} 
It is clear that $\MF(R,W)$, $\BMF(R,W)$ and $\ZMF(R,W)$ admit double
direct sums (and hence all finite direct sums of at least two
elements). On the other hand, $\HMF(R,W)$ is an additive category. Two
matrix factorizations $a_1$ and $a_2$ of $W$ over $R$ are called {\em
strongly isomorphic} if they are isomorphic in the category
$\zmf(R,W)$. It is clear that two strongly isomorphic factorizations
are also isomorphic in $\hmf(R,W)$, but the converse need not hold.
Matrix factorizations for which $M=R^{\rho|\rho}$ form a dg
subcategory of $\MF(R,W)$ which is dg-equivalent with $\MF(R,W)$. We
will often tacitly identify $\MF(R,W)$ with this subcategory.  Given
two matrix factorizations $a_1=(R^{\rho_1|\rho_1}, D_1)$ and
$a_2=(R^{\rho_2|\rho_2}, D_2)$ of $W$ with
$D_i=\left[\begin{array}{cc} 0 & v_i\\ u_i & 0\end{array}\right]$ and
$u_i,v_i\in \Mat(\rho_i,\rho_i,R)$, a morphism $f\in
\Hom_{\mf(R,W)}(a_1,a_2)$ has matrix form:
\be
f=\left[\begin{array}{cc}
    f_{\0\0} & 0\\ 0 & f_{\1\1}\end{array}\right]
\ee
with $f_{\0\0}, f_{\1\1}\in \Mat(\rho_1,\rho_2,R)$ and we have:
\be
\fd_{a_1,a_2}(f)=D_2\circ f-f\circ D_1=\left[\begin{array}{cc} 0 & v_2\circ f_{\1\1}-f_{\0\0}\circ v_1\\ u_2\circ f_{\0\0}-f_{\1\1}\circ u_1 &
    0\end{array}\right]~~.
\ee

\subsection{The triangulated structure of $\hmf(R,W)$}

The category $\hmf(R,W)$ is naturally triangulated with
an involutive suspension functor. This triangulated structure is
defined as follows (see \cite{Langfeldt} for a detailed treatment).

\

\begin{Definition}
Let $a=(M,D)$ be a matrix factorization of $W$, where 
\be
D=\left[\begin{array}{cc} 
0 & v \\
u & 0
\end{array}\right] ~~.
\ee
The \emph{suspension} of $a$ is the matrix factorization $\Sigma
M\eqdef (M', D')$, where:
\be
(M')^{\0}\eqdef M^\1~~,~~(M')^{\1}\eqdef M^\0, 
\ee
and: 
\be
D'\eqdef \left[\begin{array}{cc} 0 & -u \\
-v & 0
\end{array}\right]~~.
\ee
Given two matrix factorizations $a_1=(M_1,D_1)$ and $a_2=(M_2,D_2)$ of $W$
and a morphism $f\in \Hom_{\hmf(R,W)}(a_1,a_2)$, its
suspension $\Sigma f$ coincides with $f$ when the latter is viewed as
an element of $\Hom_R^\0(M_1',M_2')$.
\end{Definition}

\

\noindent It is easy to  check that $\Sigma$ is an endofunctor of $\hmf(R,W)$ which satisfies 
$\Sigma^2=\id_{\hmf(R,W)}$. 

\

\begin{Definition}
\label{def:cone}
Let $a_i = (M_i,D_i)$ for $i\in \{1,2\}$ be two matrix factorizations of $W$
with $D_i=\left[\begin{array}{cc} 0 & v_i \\ u_i & 0
\end{array}\right]$ and $f:a_1\rightarrow a_2$ be a morphism 
in $\hmf(R,W)$ with $f=\left[\begin{array}{cc}
    f_{\0\0} & 0\\ 0 & f_{\1\1}\end{array}\right]$. Then:
\begin{itemize}
\item The \emph{mapping cone} $C(f)$ of $f$ is the matrix
  factorization $C(f)=(M,D)$ of $W$, where:
\be
M\eqdef M^{\0} \oplus M^{\1}~~\mathrm{with}~~ M^{\0}\eqdef M_1^{\1} \oplus M_2^{\0}~~,~~M^{\1}\eqdef M_1^{\0} \oplus M_2^{\1}
\ee
and:
\be
D\eqdef \left[\begin{array}{cc} 0 & v \\u & 0 \end{array}\right]~,~\mathrm{with}~
u\eqdef \left[\begin{array}{cc}  -v_1 & 0 \\ f_{\1\1} & u_2 \end{array}\right]~,~ v\eqdef \left[\begin{array}{cc} -u_1 & 0 \\f_{\0\0} & v_2 \end{array}\right]~~.
\ee
\item The morphism $\varphi_f:a_2 \to C(f)$ is defined via the following diagram:
\be
\begin{CD}
M_2^{\0} @>u_2>> M_2^{\1} @>v_2>>  M_2^{\0}\\
@VV\iota_1V @VV\iota_2V @VV\iota_1V \\
M_1^{\1} \oplus M_2^{\0} @>u>> M_1^{\0} \oplus M_2^{\1} @>v>> M_1^{\1} \oplus M_2^{\0}
\end{CD}~~,
\ee
where $\iota_1: M_2^{\0} \to M_1^{\1} \oplus M_2^{\0}$ and 
$\iota_2: M_2^{\1} \to M_1^{\0} \oplus M_2^{\1}$ are the inclusions.
\item The morphism $\psi_f:C(f) \to \Sigma a_1$ is defined via the following diagram:
\be
\begin{CD}
M_1^{\1} \oplus M_2^{\0} @>u>> M_1^{\0} \oplus M_2^{\1} @>v>> M_1^{\1} \oplus M_2^{\0}\\
@VV\pi_1V @VV\pi_2V @VV\pi_1V \\
M_1^{\1} @>-v>> M_1^{\0}  @>-u>> M_1^{\1} 
\end{CD}~~,
\ee
where $\pi_1$ and $\pi_2$ are the natural projections.
\end{itemize}
\end{Definition}

\

\noindent The following result is well-known (see \cite{Langfeldt} for details):

\

\begin{Theorem}
\label{distinguished}
The category $\hmf(R,W)$ is triangulated when equipped with the
suspension functor $\Sigma$ and with the collection of distinguished
triangles given by sequences isomorphic with those of the form:
\be
a_1 \stackrel{f}{\lto} a_2  \stackrel{\varphi_f}{\lto} C(f)  \stackrel{\psi_f}{\lto} \Sigma a_1~,
\ee
where $f:a_1\rightarrow a_2$ is any morphism in $\hmf(R,W)$. 
\end{Theorem}

\

\begin{Proposition}
\label{prop:uniteq}
Let $s$ be a unit of $R$. Then there exists a triangulated equivalence: 
\be
\hmf(R,sW)\simeq \hmf(R,W)  ~~.
\ee
\end{Proposition}

\begin{proof}
Let $\Phi_s:\zmf(R,W)\rightarrow \zmf(R,s W)$ be the functor which
takes a factorization $a=(R^{\rho|\rho},D)$ of $W$ with
$D=\left[\begin{array}{cc} 0 & v\\ u & 0\end{array}\right]$ into the
factorization $\Phi_s(a)=(R^{\rho|\rho},D^s)$ of $sW$, where
$D^s=\left[\begin{array}{cc} 0 & sv\\ u & 0\end{array}\right]$ is a
factorization of $sW$ and leaves unchanged the morphism
$f=\left[\begin{array}{cc} f_{\0\0} & 0\\ 0 &
    f_{\1\1}\end{array}\right]$ from $a_1$ to $a_2$ into itself. Using the explicit expression:
\ben\label{first}
\fd_{a_1,a_2}^s(f)=D_2^s \circ f-f\circ D_1^s=\left[\begin{array}{cc} 0 & s v_2\circ f_{\1\1}- s f_{\0\0}\circ v_1\\ u_2\circ f_{\0\0}-f_{\1\1}\circ u_1 &
    0\end{array}\right]~~,
\een
we conclude that:
\be
D_2 \circ f-f\circ D_1 =0 \iff D_2^s \circ f-f\circ D_1^s =0 ~~.
\ee
This implies that the functor $\Phi_s$ is well-defined and\footnote{Notice that the right hand side is always a subset of the left hand
side for any element $s \in R$.  The equality holds since $s$ is a unit.}:
\be
\Hom_{\zmf(R,W)}(a_1,a_2)=\Hom_{\zmf(R,sW)}(\Phi_s(a_1),\Phi_s(a_2)) ~~.
\ee
The coboundary categories $\bmf(R,W)$ and $\bmf(R,sW)$ are also
related to each other in a similar way. More precisely, equation
\eqref{first} gives:
\be
\fd_{a_1,a_2}^s(f)=0 \iff \fd_{a_1,a_2}(f)=0 ~~,
\ee 
which implies the equality $\bmf(R,W)(a_1,a_2)=\bmf(R,sW)(\Phi_s(a_1),\Phi_s(a_2))$.  As a
result, the functor $\Phi_s$ gives an equivalence of categories from
$\hmf(R,W)$ to $\hmf(R,sW)$. Since the modules of morphisms naturally
coincide, we also conclude that $\Phi_s$ maps 
distinguished triangles into distinguished triangles. This follows
immediately from what we proved here and from Theorem \ref{distinguished}.
\qed
\end{proof}

\subsection{Localizations}

Let $S\subset R$ be a multiplicative subset of $R$ containing the
identity $1 \in R$. Let $\lambda_S:R \to R_S$ denote the natural ring
morphism from $R$ to the localization $R_S\eqdef S^{-1}R$ of $R$ at
$S$. For any $r\in R$, let $r_S\eqdef \lambda_S(r)=\frac{r}{1}\in R_S$
denote its extension. For any $R$-module $N$, let
$N_S=S^{-1}N=N\otimes_R R_S$ denote the localization of $N$ at
$S$. For any morphism of $R$-modules $f:N \rightarrow N' $, let
$f_S\eqdef f\otimes_R \id_{R_S}:N_S\rightarrow N'_S$ denote the
localization of $f$ at $S$. For any $\Z_2$-graded $R$-module
$M=M^\0\oplus M^\1$, we have $M_S=M^\0_S\oplus M^\1_S$, since the
localization functor is exact. In particular, localization at $S$
induces a functor from the category of $\Z_2$-graded $R$-modules to
the category of $\Z_2$-graded $R_S$-modules.

Let $a=(M,D)$ be a matrix factorization of $W$. The {\em localization
of $a$ at $S$} (see \cite{elbezout}) is the following matrix
factorization of $W_S$ over the ring $R_S$:
\be
a_S\eqdef (M_S,D_S)\in \Ob\MF(R_S,W_S)~~.
\ee
This extends to an even dg functor $\loc_S:\MF(R,W)\rightarrow
\MF(R_S,W_S)$, which is $R$-linear and preserves direct sums. In turn,
the latter induces functors $\ZMF(R,W)\rightarrow \ZMF(R_S,W_S)$,
$\BMF(R,W)\rightarrow \BMF(R_S,W_S)$, $\HMF(R,W)\rightarrow
\HMF(R_S,W_S)$ and $\hmf(R,W)\rightarrow \hmf(R_S,W_S)$, which we
again denote by $\loc_S$. 

\subsection{Critically-finite elements}

Since $R$ is a B\'ezout (and hence a GCD) domain, the irreducible
elements of $R$ are prime, which implies that any factorizable element
(i.e. an element with finite factorization into irreducibles) of $R$
has a unique prime factorization up to association. A divisor $d$ of
the element $W\in R^\times$ which is not a unit is called {\em
critical} if $d^2|W$. The {\em critical ideal} $\fI_W$ of $W$ is the
ideal consisting of all elements of $R$ which are divisible by every
critical divisor of $W$:
\ben
\label{fIW}
\fI_W\eqdef \{r\in R \, | \, d|r~\forall d\in R^\times~\mathrm{such~that}~d^2|W\}
\een
The following notion was introduced in \cite{elbezout}:

\

\begin{Definition}\label{def:crit-fin}
A non-zero non-unit $W$ of $R$ is called: 
\begin{itemize}
\itemsep 0.0em
\item {\em non-critical}, if $W$ has no critical divisors;
\item {\em critically-finite} if it has a factorization of the form: 
\ben
\label{Wcritform}
W=W_0 W_c~~\mathrm{with}~~W_c=p_1^{n_1}\ldots p_N^{n_N}~~,
\een
where $N\geq 1$, $n_j\geq 2$, $p_1,\ldots, p_N$ are critical prime
divisors of $W$ with $(p_i)\neq (p_j)$ for $i\neq j$ and $W_0$ is
non-critical and coprime with $W_c$.
\end{itemize}
\end{Definition} 
The elements $W_0$, $W_c$ and $p_i$ in the
factorization \eqref{Wcritform} are determined by $W$ up to
association, while $n_i$ are uniquely determined by $W$.

\begin{remark}
Let $W$ be a critically-finite element of $R$ with decomposition
\eqref{Wcritform}. Then the Chinese remainder theorem gives an
isomorphism of rings:
\be
R/\langle W\rangle \simeq R/\langle W_0\rangle \oplus R/\langle W_c\rangle~~.
\ee
When $R$ is a B\'ezout domain, the ring: 
\be
R/\langle W_c\rangle \simeq R/\langle p_1^{n_1}\rangle \oplus \ldots \oplus R/\langle p_N^{n_N}\rangle \simeq R/\langle p_1^{n_1}\ldots p_N^{n_N}\rangle
\ee
is Artinian and Gorenstein since $R/\langle p_i^{n_i}\rangle$ are Gorenstein
Artinian rings (see Section \ref{sec:artinian}). However, the rings
$R/\langle W_0\rangle$ and $R/\langle W\rangle$ need not be Noetherian.
\end{remark}

\subsection{Elementary matrix factorizations}

\noindent A matrix factorization $a=(M,D)$ of $W$ over $R$ is called
{\em elementary} if it has unit reduced rank, i.e. if $\rho(a)=1$.
Any elementary factorization is strongly isomorphic to one of the form
$e_v\eqdef (R^{1|1},D_v)$, where $v$ is a divisor of $W$ and
$D_v\eqdef \left[\begin{array}{cc} 0 & v\\ u & 0 \end{array}\right]$,
with $u\eqdef W/v\in R$. Let $\EF(R,W)$ denote the full subcategory of
$\MF(R,W)$ whose objects are the elementary factorizations of $W$ over
$R$. Let $\ZEF(R,W)$ and $\HEF(R,W)$ denote respectively the cocycle
and total cohomology categories of $\EF(R,W)$.  We also use the
notations $\zef(R,W)\eqdef \ZEF^\0(R,W)$ and $\hef(R,W)\eqdef
\HEF^\0(R,W)$ for the subcategories obtained by keeping only the even
morphisms. An elementary factorization is indecomposable in
$\zmf(R,W)$, but it need not be indecomposable in $\hmf(R,W)$.

\section{Finitely-generated modules over the quotient of a B\'ezout domain by a principal primary ideal}
\label{sec:artinian}

\noindent Let $R$ be a B\'ezout domain and $p\in R$ be a prime
element. In this section, we study the category of finitely-generated
modules over the quotient ring $R/\langle p^n\rangle$ (with $n\geq 2$)
and its stable category.

\subsection{The rings $A_n(p)$}

Fix an integer $n\geq 2$ and consider the quotient ring\footnote{This ring will later on also be 
denoted by $\Lambda$ for ease of notation.}:
\be
A_n(p)\eqdef R/\langle p^n\rangle~~.
\ee
Let $\bm_n(p)=p A_n(p)=\langle p\rangle /\langle p^n\rangle$ and $\bk_p=R/\langle p\rangle$.
The following result was proved in \cite{elbezout}.

\

\begin{Lemma} 
\label{lemma:prime}
The following statements hold:
\begin{enumerate}
\itemsep 0.0em
\item The principal ideal $\langle p\rangle $ generated by $p$ is maximal.
\item The primary ideal $\langle p^n\rangle $ is contained in a unique maximal ideal
  of $R$.
\item The quotient $A_n(p)$ is a quasi-local ring with maximal ideal
  $\bm_n(p)$ and residue field $\bk_p$.
\item $A_n(p)$ is a generalized valuation ring.
\end{enumerate}
\end{Lemma}

\

\begin{remark}
Let $Z(A_n(p))$ be the set of zero divisors, $N(A_n(p))$ be the
nilradical and $J(A_n(p))$ be the Jacobson radical of $A_n(p)$. Then
we have (see \cite[Exercise 1.1]{FS}):
\be
Z(A_n(p))=N(A_n(p))=J(A_n(p))=\bm_n(p)~~.
\ee
\end{remark}

\

\begin{Proposition}
\label{prop:Artinian}
$A_n(p)$ is an Artinian local principal ideal ring, whose ideals are
  $\langle p^i\rangle/\langle p^n\rangle$ for $i=0,\ldots, n$. 
\end{Proposition}

\begin{proof}
Let $I$ be an ideal of $R$ such that $\langle p^n\rangle \subsetneq I
\subsetneq \langle p\rangle$. Since $A_n(p)$ is a generalized
valuation ring by Lemma \ref{lemma:prime}, its ideals are totally
ordered by inclusion. Hence there exists an $i\in \{2,\ldots, n-1\}$
such that $\langle p^i\rangle \subset I\subsetneq \langle
p^{i-1}\rangle $.  Suppose that $I\setminus \langle p^i\rangle$ is
non-empty and take any element $x\in I\setminus \langle
p^i\rangle$. Then $x=rp^{i-1}$ for some $r\in R$ such that $p$ doesn't
divide $r$, i.e. $(r,p)=(1)$. Since $R$ is a B\'ezout domain, there exist
$a,b\in R$ such that $ar+bp=1$. Multiplying with $p^{i-1}$, this gives
$p^{i-1}=a x+bp^i$, which belongs to $I$ since both $x$ and $p^i$
belong to $I$. Thus $p^{i-1}\in I$, which implies $\langle
p^{i-1}\rangle\subset I$ and hence $I=\langle p^{i-1}\rangle$,
contradicting the fact that the inclusion $I\subset \langle
p^{i-1}\rangle $ is strict. It follows that every ideal of $R/\langle
p^n\rangle$ has the form $\langle p^i\rangle /\langle p^n\rangle$ for
some $i\in \{0,\ldots, n\}$. In particular, $R/\langle p^n\rangle$ is
an Artinian (and hence Noetherian) local ring. Since $R$ is a
Noetherian B\'ezout ring, it is also a principal ideal ring.\qed
\end{proof}

\

\begin{remark}
Since $A_n(p)$ has non-trivial divisors of zero, it cannot be a
regular local ring.  It was shown in {\rm \cite{Osofsky}} that the global
dimension of a generalized valuation ring which is not an integral
domain is necessarily infinite. Thus $\mathrm{gl}\dim
(A_n(p))=\infty$.  Also notice that $A_n(p)$ has length $n$ as a
module over itself.
\end{remark}

\

\noindent For simplicity, in the remainder of this section we denote
$A_n(p)$ by $\Lambda$, the residue field $\bk_n(p)$ by $\bk$ and the
maximal ideal $\bm_n(p)$ by $\bm$.

\subsection{The category $\mod_\Lambda$}

Let $\mod_\Lambda$ be the category of finitely-generated modules over
$\Lambda=A_n(p)$. Since $\Lambda$ is Artinian, the following statements are
equivalent for a $\Lambda$-module $M$ by the Akizuki-Hopkins-Lewitzki
theorem:
\begin{itemize}
\itemsep 0.0em
\item $M$ is Noetherian. 
\item $M$ is Artinian. 
\item $M$ is finitely-generated.
\item $M$ has finite composition length.
\end{itemize}

\noindent Let $\Lambda_i=\langle p^{n-i}\rangle /\langle
p^n\rangle=p^{n-i}\Lambda$ with $i\in \{0,\ldots, n\}$ be the ideals
of $\Lambda$, thus $\Lambda_0=0$, $\Lambda_{n-1}=\bm$ and
$\Lambda_n=\Lambda$. These form the finite ascending sequence:
\ben
\label{LambdaSeq}
0=\Lambda_0\subset \Lambda_1\subset \ldots \subset \Lambda_{n-1}\subset \Lambda_n=\Lambda~~.
\een
Let $V_i\eqdef \Lambda/\Lambda_{n-i}\simeq_R R/\langle p^i\rangle$
(with $i=0, \ldots, n$) be the cyclically-presented cyclic
$\Lambda$-modules with annihilators $\Ann(V_i)=\Lambda_{n-i}$. We have
natural isomorphisms of $R$-modules
$\varphi_i:V_i\stackrel{\sim}{\rightarrow} \Lambda_i$ given by taking
the element $x+\langle p^i\rangle$ ($x\in R$) of $V_i\simeq_R
R/\langle p^i\rangle$ to the element $p^{n-i}x+\langle p^n\rangle$ of
$\Lambda_i$. Unlike the ideals $\Lambda_i$ (which can be viewed as
{\em non-unital} $\Lambda$-algebras), the modules $V_i$ have a {\em
unital} $\Lambda$-algebra structure with unit
$1_{\Lambda}+\Lambda_{n-i}$. This unit is not preserved by the
$R$-module isomorphisms $\varphi_i$.  It is clear that the non-zero
cyclic modules $V_1,\ldots, V_n$ are indecomposable, with endomorphism
rings given by the local rings:
\be
\End_\Lambda(V_i)\simeq R/\langle p^i\rangle~~,~~\forall i\in \{1,\ldots, n\}~~.
\ee
Recall that a commutative ring $R$ is called an {\em $FGC$ 
(finitely-generated commutative) ring} if
every finitely-generated $R$-module is isomorphic with a finite direct
sum of cyclic modules. For any FGC ring $R$, the finite direct sum
decomposition of a finitely-generated $R$-module into non-zero
indecomposable cyclic modules is unique up to permutation and
isomorphism of the indecomposable cyclic summands \cite{Brandal}.

\

\begin{Proposition}
$\Lambda$ is an FGC ring whose indecomposable non-zero
  finitely-generated $\Lambda$-modules are the cyclic modules
  $V_1,\ldots, V_n$.  Moreover, the decomposition of a
  finitely-generated $\Lambda$-module into non-zero cyclic modules is
  unique up to permutation and isomorphism of factors, hence $\mod_\Lambda$ 
  is a Krull-Schmidt category. 
\end{Proposition}

\begin{proof}
It is well-known that {\em any} module over a principal ideal ring
decomposes as a direct sum of cyclic modules \cite{Kothe}. In
particular, $\Lambda$ is an FGC ring. Uniqueness of the decomposition
into non-zero cyclic modules up to permutation and isomorphism of
factors follows from \cite[Proposition 3.4]{Brandal} since $\Lambda$
is a generalized valuation ring. The indecomposable
finitely-generated $\Lambda$-modules coincide with the cyclic modules
$V_1,\ldots, V_n$. See \cite[Theorem 3.2]{Wiegand}. \qed
\end{proof}

\

\begin{Proposition}
\label{prop:indproj}
The only non-zero indecomposable $\Lambda$-module which is projective
is $V_n\!\simeq\! \Lambda_n\!\!=\!\Lambda$.
\end{Proposition}

\begin{proof}
Since any projective module over a local ring is free, it
follows that a finitely-generated $\Lambda$-module is projective iff
it is free of finite rank. Such a module is indecomposable iff it has rank one. 
Another way to see this is as follows. 
Since the non-zero indecomposable $\Lambda$-modules are $V_i$ with
$i\in \{1,\ldots, n\}$, it suffices to show that $V_i$ is projective iff
$i=n$.  The module $\Lambda_n=\Lambda$ is projective since it is
free. Thus it suffices to show that $V_1,\ldots, V_{n-1}$ are not
projective.  Recall that $\Lambda_{n-i}=p^{i}\Lambda$ is a principal
$\Lambda$-module. It is well-known that such a module is projective
iff there exists an idempotent $e\in \Lambda$ such that
$p^i\Lambda=e\Lambda$. Suppose that this is the case for some
$i\in \{1,\ldots, n-1\}$.  Then we must have:
\ben
\label{idemp}
p^{2i}\Lambda=e^2\Lambda=e\Lambda=p^i \Lambda~~.  
\een
If $2i\leq n$, this amounts to $\Lambda_{2i}=\Lambda_i$, which is
impossible since the inclusions in \eqref{LambdaSeq} are strict. If
$2i\geq n$, then we have $p^{2i}\Lambda=0$ and relation \eqref{idemp}
amounts to $p^i\Lambda=0$, which is impossible since $i$ belongs to
the set $\{1,\ldots, n-1\}$.
\qed
\end{proof}

\subsection{Uniseriality}

Notice that $\Lambda$ is a uniserial ring and that the indecomposable
cyclic modules $V_i\simeq_\Lambda \Lambda_i$ are uniserial modules of length
$i$. The unique composition series of $\Lambda_i$ is given by:
\be
0=\Lambda_0\subset \ldots \subset \Lambda_i~~. 
\ee
In particular, the only simple $\Lambda$-module is $\Lambda_1\simeq_R V_1\simeq_R \bk$. 
We have: 
\be
V_{i+1}/V_i\simeq \Lambda_{i+1}/\Lambda_i\simeq \bk
\ee 
and the only composition factor of $\Lambda_i\simeq_\Lambda V_i$ is $\bk$,
with multiplicity $i$.

\subsection{The Frobenius property} The following result shows that $\Lambda$ is 
a Frobenius ring. 

\

\begin{Proposition}
The ring $\Lambda$ is a commutative Frobenius ring. In particular, $\Lambda$ is
  self-injective and hence it is a Gorenstein ring of dimension zero. Thus:
\be
\Ext_\Lambda^i(\bk,\Lambda)\simeq_\Lambda \twopartdef{\bk~}{i=0}{0~}{i\neq 0}~~.
\ee
\end{Proposition}

\

\begin{proof}
It is clear that $\Lambda$ has a unique minimal ideal, namely
$\Lambda_1$. Since $\Lambda$ is a local Artinian ring, it follows that
$\Lambda$ is Frobenius. This implies that $R$ is self-injective and
hence Gorenstein of dimension zero. \qed
\end{proof}

\

\noindent Since $\Lambda$ is Noetherian and self-injective
(i.e. quasi-Frobenius, which for a commutative ring is the same as
being Frobenius), it follows that a $\Lambda$-module is injective iff
it is projective. In particular, $\mod_\Lambda$ is a Frobenius
category. Notice that $K_\Lambda=\Lambda$ is a canonical $\Lambda$-module. In
particular, all finitely-generated $\Lambda$-modules are reflexive.

\subsection{The Auslander-Reiten quiver of $\mod_\Lambda$}
\label{subsec:ARquiver}

The following result allows us to describe the morphisms between the modules $V_i$. 

\

\begin{Proposition}
Let $R$ be a B\'ezout domain and $a,b\in R^\times$. Then there
exists an isomorphism of $R$-modules:
\be
q_{ab}:\Hom_R(R/\langle a\rangle ,R/\langle b\rangle)\stackrel{\sim}{\rightarrow} R/\langle a,b\rangle~~
\ee
which is determined up to multiplication by a unit of $R$. 
If $a,b,c\in R^\times$ are three elements and $f\in
\Hom_R(R/\langle a \rangle,R/\langle b\rangle)$, $g\in \Hom_R(R/\langle b\rangle,R/\langle c\rangle)$, then we have:
\be
q_{ac}(g\circ f) \eqdef s_{abc} q_{bc}(g)q_{ab}(f)~~,
\ee
where $s_{abc}\in \frac{(a,c)(b)}{(b,c)(a,b)}$. 
\end{Proposition}

\begin{proof}
The cyclic module $R/\langle a\rangle$ is generated by the element
$\epsilon_a=1 \,\mod \, \langle a\rangle$, while $R/\langle b\rangle$ is
generated by $\epsilon_b=1\,\mod \, \langle b\rangle$.  Consider the
injective $R$-module morphism $\varphi_{ab}:\Hom_R(R/\langle
a\rangle,R/\langle b\rangle)\rightarrow R/\langle b\rangle$ which
associates to $f\in \Hom_R(R/\langle a\rangle,R/\langle b\rangle)$ the
unique element $\varphi_{ab}(f)\in R/\langle b\rangle$ such that
$f(\epsilon_a)=\varphi_{ab}(f)\epsilon_b$. Let $r\in R$ be a element
such that $\varphi_{ab}(f)= r\, \mod \langle b\rangle$.  Since
$a\varphi_{ab}(f) \epsilon_b=af(\epsilon_a)=f(a\epsilon_a)=f(0)=0$, we
have $a\varphi_{ab}(f)=0$ in the ring $R/\langle b\rangle$, which is
equivalent with the condition $b|ar$. Writing $a=a_1 d_{ab}$ and $b=b_1 d_{ab}$
with $d_{ab}\in (a,b)$ and $(a_1,b_1)=(1)$, this is equivalent with the
condition $b_1|r$, i.e. $r\in \langle b_1\rangle$. Hence the image of
$\varphi_{ab}$ equals $\langle b_1\rangle/\langle b\rangle$. The map
$\langle b_1\rangle \ni r\rightarrow r/b_1 \in R$ induces an
isomorphism of $R$-modules $\psi_{ab}:\langle b_1\rangle /\langle
b\rangle \stackrel{\sim}{\rightarrow} R/\langle a,b\rangle$. Then
$q_{ab}\eqdef \psi_{ab}\circ \varphi_{ab}: \Hom_R(R/\langle a\rangle
,R/\langle b\rangle)\rightarrow R/\langle a,b\rangle$ is the desired
isomorphism of $R$-modules, which acts as
$q_{ab}(f)=\frac{\varphi_{ab}(f)}{b_1}=\frac{d_{ab}\varphi_{ab}(f)}{b}$.
Since $d_{ab}$ is determined up to multiplication by a unit of $R$, the same 
holds for $q_{ab}(f)$. 

Given three non-vanishing elements $a,b,c$ of $R$ and morphisms $f,g$
as in the proposition, we have:
\be
(g\circ f)(\epsilon_a)=g(\varphi_{ab}(f)\epsilon_b)=\varphi_{bc}(g)\varphi_{ab}(f) \epsilon_c~~,
\ee
which gives $\varphi_{ac}(g\circ f)=\varphi_{bc}(g)\varphi_{ab}(f)$. Thus: 
\be
q_{ac}(g\circ f)=\frac{d_{ac}\varphi_{bc}(g)\varphi_{ab}(f)}{c}=\frac{d_{ac} c b}{cd_{bc}d_{ab}} q_{bc}(g)q_{ab}(f)=s_{abc} q_{bc}(g)q_{ab}(f)~,
\ee
where: 
\be
s_{abc}=\frac{d_{ac} b}{d_{bc}d_{ab}} q_{bc}(g)q_{ab}(f)\in \frac{(a,c)(b)}{(b,c)(a,b)}~~.
\ee
\end{proof}

\begin{Corollary}
\label{cor:choms}
Let $R$ be a B\'ezout domain and $a,b, c\in R^\times$ be three
elements such that $a|c$ and $b|c$.  Then there exists an
isomorphism of $R/\langle c\rangle$-modules:
\be
\Hom_{R/\langle c\rangle}(R/\langle a\rangle ,R/\langle b\rangle )\simeq R/\langle a,b\rangle~~.
\ee
\end{Corollary}

\begin{proof}
Restriction of scalars along the epimorphism $\pi:R\rightarrow R/\langle c\rangle$
gives a full and faithful functor $\pi^\ast:\Mod_{R/\langle c\rangle}\rightarrow
\Mod_R$. The composition $q_{ab}\circ \pi_{R/\langle a\rangle
,R/\langle b\rangle }^\ast:\Hom_{R/\langle c\rangle}(R/\langle
a\rangle,R/\langle b\rangle)\rightarrow R/\langle a,b\rangle$ is the
desired isomorphism. \qed
\end{proof}

\

\begin{Proposition}
\label{prop:Hom_Lambda}
For any $i,j\in \{0,\ldots, n\}$, we have an isomorphism of modules: 
\be
\Hom_\Lambda(V_i,V_j)\simeq_\Lambda V_{\min(i,j)}~~.
\ee
For any $i\in \{1,\ldots, n\}$, we have an isomorphism of rings: 
\be
\End_\Lambda(V_i)\simeq R/\langle p^i\rangle~~.
\ee
\end{Proposition}

\begin{proof}
Follows immediately from Corollary \ref{cor:choms}. \qed
\end{proof}

\

\noindent In particular, $\End_\Lambda(V_i)$ is a commutative local
ring with maximal ideal $m_i\eqdef \langle p\rangle /\langle
p^i\rangle$ and residue field equal to $\bk_p$. Consider the field:
\be
T(V_i)\eqdef \End_\Lambda(V_i,V_i)/m_i\simeq \bk_p~~.
\ee

\begin{Proposition}
For any $0\leq j\leq i \leq n$, we have: 
\be
V_i/V_j\simeq_\Lambda V_{i-j}~~.
\ee
Moreover, the natural surjection $q_{n,i}:V_n\rightarrow V_i$ is a
projective cover for all $i\in \{1,\ldots, n\}$ and the first syzygy of $V_i$
is given by: \be \Omega(V_i)=\ker(q_{n,i})\simeq V_{n-i} ~~.  \ee
\end{Proposition}

\begin{proof}
We have $V_i/V_j=\langle p^{n-i}\rangle/\langle p^{n-j}\rangle\simeq_R
R/\langle p^{i-j}\rangle=V_{i-j}$, so similar isomorphisms hold over
$\Lambda$. Recall that $V_n\simeq \Lambda$ is a projective module.
Since each $V_k$ has a single maximal submodule (namely $V_{k-1}$), we
have $\rad(V_k)=V_{k-1}$ for all $k\in \{1,\ldots, n\}$.  The induced
map $\bar{q}_{n,i}:V_n/\rad(V_n)\rightarrow V_i/\rad(V_i)$ is an
isomorphism since $V_n/V_{n-1}\simeq_\Lambda R/\langle p\rangle
\simeq_\Lambda V_i/V_{i-1}$.  This implies that $q_{n,i}$ is a
projective cover by \cite[Chap I.4, Proposition 4.3, page 13]{AR}. It
is clear that $\ker(q_{n,i})\simeq V_{n-i}$. \qed
\end{proof}

\

\begin{Proposition}
Let $f:V_i\rightarrow V_j$ be an irreducible morphism in
$\mod_\Lambda$. Then one of the following holds:
\begin{enumerate}
\itemsep 0.0em
\item $f$ is injective and $j=i+1$. In this case, $f$ fits into a
  short exact sequence:
\be
0\longrightarrow V_i\stackrel{f}{\longrightarrow} V_{i+1}\longrightarrow V_1\longrightarrow 0 ~~.
\ee 
\item $f$ is surjective and $j=i-1$. In this case, $f$ fits into a short exact sequence: 
\be
0\longrightarrow V_1\longrightarrow V_i\stackrel{f}{\longrightarrow} V_{i-1}\longrightarrow 0~~.
\ee
\end{enumerate}
\end{Proposition}

\begin{proof}
Recall that an irreducible morphism $f:V_i\rightarrow V_j$ in
$\mod_\Lambda$ must be either a monomorphism or an epimorphism
\cite[Chap. V.5, Lemma 5.1]{AR}.  Distinguish the cases:
\begin{enumerate}
\itemsep 0.0em
\item If $f$ is a monomorphism, then $\im f=V_k$ for some $k\leq
  j$. Since $V_i\simeq_\Lambda \im f$, we must have $k=i$ and $\im
  f=V_i$. Thus $i\leq j$. Moreover, $\im f$ must be a direct summand
  of any proper submodule of $V_j$ which contains $\im f$. Since no
  submodule of $V_j$ has a direct summand, we must have $\im
  f=V_{j-1}$ and hence $j=i+1$.
\item If $f$ is an epimorphism, then $\ker f=V_k$ for some $k\leq
  i$. Since $V_j\simeq_\Lambda V_i/\ker f=V_i/V_k\simeq_\Lambda V_{i-k}$, we must have
  $i\geq j$ and $k=i-j$. Moreover, $V_j$ must be a summand of $V_i/M$
  for any non-zero submodule $M$ of $V_i$ which is contained in $\ker
  f=V_k$, i.e. it must be a summand of $V_i/V_s=V_{i-s}$ for any
  $s\in \{1,\ldots, k\}$. Since none of the modules $V_1,\ldots, V_n$ has
  direct summands, this means that we must have $k=1$,
  i.e. $i=j+1$. 
\end{enumerate}
The short exact sequences follow immediately from the above. \qed
\end{proof} 

\

\noindent For any $i\in \{1,\ldots, n-1\}$, let $s_{i,i+1}:V_i\rightarrow
V_{i+1}$ be the inclusion. For any $i=2,\ldots, n$, let
$q_{i,i-1}:V_i\rightarrow V_{i-1}$ be the natural surjection. For any
$i\in \{1,\ldots, n-1\}$, we have an almost split sequence (see
\cite[p. 141]{AR}):
\be
0\longrightarrow V_i \stackrel{g_i}{\longrightarrow} V_{i-1}\oplus V_{i+1}\stackrel{f_i}{\longrightarrow} V_i\longrightarrow 0~~,
\ee
where $g_i=\left[\begin{array}{c} -q_{i,i-1}
    \\ s_{i,i+1}\end{array}\right]$ and $f_i=\left[\begin{array}{cc}
    s_{i-1,i} &,~q_{i+1,i}\end{array}\right]$.  In particular, the
morphisms $s_{i,i+1}$ and $q_{i,i-1}$ are irreducible by
\cite[Chap. V.5., Theorem 5.3, p. 167]{AR}. Moreover, the
Auslander-Reiten translation $\tau=\rD\Tr$ is given by:
\be
\tau(V_i)=V_i~,~~\forall i\in \{1,\ldots, n-1\}~,~~\tau(V_n)=0~~.
\ee
(recall that $\rD\Tr(P)=0$ iff $P$ is a projective module).  It follows
that $\rD\Tr$ acts trivially on $\Lambda$-modules which have no
projective direct summands. By \cite[page 229]{AR}, the class
$\bar{s}_{i-1,i}$ of $s_{i-1,i}$ generates the
$T(V_{i-1})^\opp$-vector space $\Irr(V_{i-1},V_i)$ while the class
${\bar q}_{i+1,i}$ of $q_{i+1,i}$ generates the
$T(V_{i+1})^\opp$-vector space $\Irr(V_{i+1},V_i)$. Similarly, the
class $\bar{s}_{i,i+1}$ of $s_{i,i+1}$ generates the
$T(V_{i+1})$-vector space $\Irr(V_i,V_{i+1})$ and the class
$\bar{q}_{i,i-1}$ of $q_{i,i-1}$ generates the $T(V_{i-1})$-vector
space $\Irr(V_i,V_{i-1})$. Thus:
\begin{itemize}
\itemsep 0.0em
\item $\Irr(V_i, V_{i+1})$ is generated by $\bar{s}_{i,i+1}$ over both
  $T(V_{i})^\opp$ and $T(V_{i+1})$~.
\item $\Irr(V_i,V_{i-1})$ is generated by $\bar{q}_{i,i-1}$ over both
  $T(V_{i})^\opp$ and $T(V_{i-1})$~.
\end{itemize}
It follows that the arrow $V_i\rightarrow V_{i-1}$ for $i=2,\ldots,
n-1$ and the arrows $V_{i-1}\rightarrow V_i$ have trivial valuation
$(1,1)$. The Auslander-Reiten quiver of $\mod_\Lambda$ is shown in
Figure \ref{fig:modquiver}.

\begin{figure}[H]
\centering \scalebox{0.7}{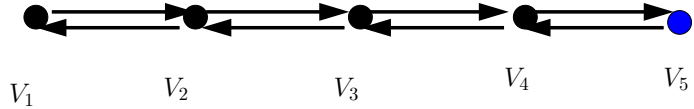}
\caption{Auslander-Reiten quiver for $\mod_\Lambda$ when $n=5$. The
  single projective injective vertex is shown in blue.  The
  Auslander-Reiten translation fixes all non-projective vertices.}
\label{fig:modquiver}
\end{figure}

\subsection{The category $\umod_\Lambda$}

Let $\umod_\Lambda$ denote the projectively-stable category of
finitely-generated $\Lambda$-modules.  Since any projective
$\Lambda$-module is free, this category has the same objects as
$\mod_\Lambda$ and morphisms given by:
\be
\uHom_\Lambda(M,N)\eqdef \Hom_\Lambda(M,N)/\cP_\Lambda(M,N)~,~~\forall M,N\in \Ob(\mod_\Lambda)~~,
\ee
where $\cP_\Lambda(M,N)\subset \Hom_\Lambda(M,N)$ is the submodule
consisting of those morphisms from $M$ to $N$ which factor through a
free module of finite rank. Since $\mod_\Lambda$ is a Frobenius category, 
the stable category $\umod_\Lambda$ has a natural triangulated structure. 

The first syzygy induces a functor $\Omega:\umod_\Lambda\rightarrow
\umod_\Lambda$ which is an equivalence of categories since $\Lambda$
is self-injective (see \cite[Chap. IV.3]{AR}). Since $\Lambda$ is a
symmetric Artin algebra, we also have $\rD\simeq
\Hom_\Lambda(-,\Lambda)$ and $\Omega^2\simeq \rD\Tr=\tau$. Since $\rD\Tr$
acts as the identity functor on indecomposable non-projectives of
$\mod_\Lambda$, we have $\rD\Tr\simeq \id_{\umod_\Lambda}$ and hence
$\Omega^2\simeq \id_{\umod_\Lambda}$. The functor $\Omega$ is the
shift functor of the triangulated category $\umod_\Lambda$.

\

\noindent For $i,j\in \{0,\ldots,n\}$, define: 
\ben
\label{deltamu}
\delta_n(i)\!\eqdef\! \min(i,n-i)\in \{1,\ldots, n-1\}~,~~
\mu_n(i,j)\!\eqdef\!\min[\delta_n(i),\delta_n(j)]=\fourpartdef{i}{i+j\leq n 
~\&~ i\leq j}{j}{i+j\leq n ~\&~ i>j }{n-i}{i+j>n ~\&~ i>j}{n-j}{i+j>n ~\&~ i\leq j}.
\een
Notice the relations $\delta_n(i)=\delta_n(n-i)$ and $\delta_n(n)=0$ as well as:
\ben
\label{murels}
\mu_n(i,j)=\mu_n(j,i)=\mu_n(n-i,j)=\mu_n(i,n-j)~,~~\mu_n(i,n)=0~~.
\een

\begin{Proposition}
\label{prop:umod}
For any $1\leq i,j\leq n-1$, we have: 
\be
\uHom_\Lambda(V_i,V_j)\simeq_{\Lambda} V_{\mu_n(i,j)}~~.
\ee
\end{Proposition}

\begin{proof}
A similar statement is proved in \cite[Lemma 2.3]{CSZ}.  For
completeness we sketch the proof. Proposition \ref{prop:Hom_Lambda}
gives an isomorphism of $\Lambda$-modules:
\be
\Hom_{\Lambda}(V_i, V_j) \simeq_\Lambda V_{\min(i,j)}\simeq_\Lambda p^{n-\min(i,j)}\Lambda =p^{\max(n-i,n-j)}\Lambda=(p^{n-i}\Lambda)\cap (p^{n-j}\Lambda)~~,
\ee
where we noticed that $n-\min(i,j)=\max(n-i,n-j)$.  The morphism $f\in
\Hom_{\Lambda}(p^{n-i}\Lambda, p^{n-j}\Lambda )$ factors through a
free module iff\footnote{As in \cite[Lemma 2.3]{CSZ}, this follows
from the fact that the natural morphism of modules from $V_i\simeq
p^{n-i} \Lambda$ to $V_i^\vee=\Hom_\Lambda(p^{n-i}\Lambda,
\Lambda)=\Hom_\Lambda(V_i,V_n)\simeq V_{\min(i,n)}=V_i$ is an
isomorphism by Proposition \ref{prop:Hom_Lambda}.} its image
through this isomorphism lies in the ideal $p^{n-i}\Lambda
p^{n-j}=p^{2n-i-j}\Lambda$. Thus:
\be
\uHom(V_i,V_j) \simeq_\Lambda \frac{p^{\max(n-i,n-j)}\Lambda}{p^{2n-i-j}\Lambda}\simeq_\Lambda \frac{p^{n-\min(i,j)}\Lambda}{p^{2n-i-j}\Lambda}~~.
\ee 
The denominator is isomorphic to $0$ when $i+j \leq n$. In this case we have:
\be
\uHom(V_i,V_j) \simeq_\Lambda R/\langle p^{\min(i,j)}\rangle=V_{\min(i,j)}~~.
\ee 
On the other hand, when $i+j>n$, we find:
\be
\uHom(V_i,V_j) \simeq_\Lambda \frac{\langle p^{\max(n-i,n-j)}\rangle }{\langle p^{2n-i-j}\rangle } \simeq_\Lambda R/\langle p^{\min(n-i,n-j)}\rangle=V_{\min(n-i,n-j)}~~,
\ee 
where we noticed that $2n-i-j=\min(n-i,n-j)+\max(n-i,n-j)$.  The 
conclusion follows upon noticing that:
\be
\mu_n(i,j)=\twopartdef{\min(i,j)}{i+j\leq n}{\min(n-i,n-j)}{i+j>n}~~.~~\qed
\ee
\end{proof}

\subsection{The Auslander-Reiten quiver of $\umod_\Lambda$}

For any $\Lambda$-module $M$, there exists an injective resolution: 
\be
M\longrightarrow M_0\longrightarrow M_1 \longrightarrow \ldots 
\ee
whose cohomology in degree one equals $\Omega(M)$. Hence we
have natural isomorphisms of $\Lambda$-modules
$\Ext^1(N,M)\simeq_\Lambda \uHom_\Lambda(N,\Omega(M))$ and any
Auslander-Reiten sequence:
\ben
\label{AR}
0\longrightarrow X \stackrel{f}{\longrightarrow} Y \stackrel{g}{\longrightarrow} Z\longrightarrow 0
\een
induces an Auslander-Reiten triangle: 
\be
X\stackrel{\underline{f}}{\longrightarrow} Y \stackrel{\underline{g}}{\longrightarrow} Z\stackrel{\psi}{\longrightarrow} \Omega(X)~~,
\ee
where $\psi\in \uHom_\Lambda(Z,X[1])=\Ext^1(Z,X)$ is the extension class
defined by the AR sequence \eqref{AR}. As a consequence, the category $\umod_\Lambda$ has
Auslander-Reiten triangles which are given by:
\ben
\label{ARtriangles}
V_i \stackrel{\underline{g_i}}{\longrightarrow} V_{i-1}\oplus V_{i+1}\stackrel{\underline{f_i}}{\longrightarrow} V_i\longrightarrow \Omega(V_i)~,~~\forall i\in \{1,\ldots, n-1\}~~.
\een 
In particular, $V_1,\ldots, V_{n-1}$ are indecomposable objects of
$\umod_\Lambda$ which have local endomorphism rings. It follows that
$\umod_\Lambda$ is Krull-Schmidt with indecomposables $V_1,\ldots,
V_{n-1}$. Moreover, $\underline{g_i}$ are source morphisms and
$\underline{f_i}$ are sink morphisms, which implies $\dim_{T(V_i)}
\Irr(V_i,V_{i+1})=\dim_{T(V_{i+1})}\Irr(V_i,V_{i+1})=1$ and
$\dim_{T(V_i)}
\Irr(V_i,V_{i-1})=\dim_{T(V_{i-1})}\Irr(V_i,V_{i-1})=1$ (see
\cite{Liu}). Hence all arrows of the AR quiver of $\umod_\Lambda$ have
trivial valuation $(1,1)$. The AR translation is given by $\tau(V_i)=V_i$
for all $i\in \{1,\ldots, n-1\}$. The AR quiver of $\umod_\Lambda$ is
obtained from that of $\mod_\Lambda$ by deleting the projective vertex;
an example is shown in Figure \ref{fig:squiver}.

\

\begin{figure}[H]
\centering \scalebox{0.7}{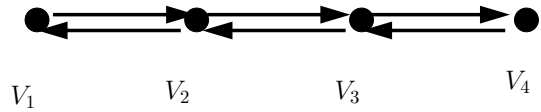}
\caption{Auslander-Reiten quiver for $\umod_\Lambda$ when $n=5$.  The 
translation fixes all vertices. }
\label{fig:squiver}
\end{figure}

\subsection{The Calabi-Yau property of $\umod_\Lambda$}

Recall that $\Lambda$ is a self-injective (a.k.a. quasi-Frobenius)
commutative ring. This implies that all finitely-generated
$\Lambda$-modules are reflexive and that the dual
$\rD(M)=\Hom_\Lambda(M,\Lambda)$ of any finitely-generated module is
finitely-generated \cite[Theorem 4.12.21]{Hazewinkel}. Thus $\rD$ is an
involutive auto-equivalence of $\umod_\Lambda$. Since $\Lambda$ is
self-injective, we have $\umod_\Lambda\simeq \overline{\mod}_\Lambda$
and hence $\rD$ induces a well-defined involutive autoequivalence of
$\umod_\Lambda$ by \cite[Chap. IV.1, Proposition 1.9, page 106]{AR},
which we denote by the same letter.

\

\begin{Lemma}
\label{lemma:duals}
We have: 
\be
\rD(V_i)\simeq_\Lambda V_i~,~~\forall i\in \{1,\ldots, n\}~~.
\ee
\end{Lemma}

\begin{proof}
For any $i\in \{1,\ldots, n\}$, we have: 
\be
\rD(V_i)=\Hom_{\Lambda}(V_i, V_n)\simeq_\Lambda V_{\min(i,n)}=V_i~~,
\ee
where we used Proposition \ref{prop:Hom_Lambda}. \qed 
\end{proof}

\

Recall that an additive autoequivalence $S$ of the $R$-linear category
$\umod_\Lambda$ is called a {\em Serre functor} if we have natural
isomorphisms of $\Lambda$-modules:
\be
\uHom_\Lambda(M,N)\simeq_\Lambda \rD(\uHom_\Lambda(N,S(M)))~,~~\forall M,N\in \Ob[\umod_\Lambda]~~.
\ee
This implies that $S$ is a triangulated
auto-equivalence. The following proposition shows
that the $R$-linear triangulated category $\umod_\Lambda$ is
``1-Calabi-Yau'':
 
\

\begin{Proposition}
The functor $S=\Omega$ is a Serre functor for $\umod_\Lambda$. 
\end{Proposition}

\begin{proof}
Since $\umod_\Lambda$ is Krull-Schmidt with indecomposable objects
$V_1,\ldots, V_{n-1}$, it suffices to show that we have natural
isomorphisms in $\umod_\Lambda$:
\ben
\label{sgen}
\uHom_\Lambda(V_i,V_j)\simeq_{\umod_\Lambda} \rD(\uHom_\Lambda(V_j,\Omega(V_i)))~,~~\forall i,j\in \{1,\ldots, n-1\}~~.
\een
Since $\Omega V_i\simeq_\Lambda V_{n-i}$, Proposition \ref{prop:umod} shows
that the right hand side of \eqref{sgen} is given by:
\be
\rD(\uHom_{\Lambda}(V_j,\Omega(V_i)))\simeq_{\Lambda} \rD(\uHom_{\Lambda}(V_j, V_{n-i}))\simeq_\Lambda \rD(V_{\mu_n(j,n-i)})= \rD(V_{\mu_n(i,j)})\simeq_\Lambda V_{\mu_n(i,j)}~~,
\ee
where we used relations \eqref{murels} and Lemma \ref{lemma:duals}. On the
other hand, the left hand side of \eqref{sgen} is given by:
\be
\uHom_\Lambda(V_i,V_j)\simeq_\Lambda V_{\mu_n(i,j)}~~.
\ee
Since all isomorphisms above are natural, we conclude that
\eqref{sgen} holds since any isomorphism in $\mod_\Lambda$ induces an
isomorphism in $\umod_\Lambda$. \qed
\end{proof}

\subsection{A triangle generator for $\umod_\Lambda$}
We say that a full subcategory $\cC$ of $\umod_\Lambda$ is {\em closed
  under extensions} (also known as {\em thick} or {\em \'epaisse}) if,
given any distinguished triangle:
\be
X\longrightarrow Y\longrightarrow Z\longrightarrow \Omega(X)
\ee
of $\umod_\Lambda$, we have $Y\in \Ob \cC$ provided that $X$ and $Z$
are objects of $\cC$. We say that a full subcategory $\cC$ of
$\umod_\Lambda$ is {\em isomorphism-closed} (or {\em strictly full})
if any object of $\umod_\Lambda$ which is isomorphic with an object of
$\cC$ is an object of $\cC$. A full subcategory $\cC$ of
$\umod_\Lambda$ is called {\em saturated} if it is closed under direct
summands. Given an object $X$ of $\umod_\Lambda$, let $\langle
X\rangle $ denote the smallest triangulated subcategory of
$\umod_\Lambda$ which contains the object $X$ and is strictly full and
saturated. This coincides with the smallest full subcategory of
$\umod_\Lambda$ which is closed under isomorphisms, direct sums,
shifts and direct summands.

\

\begin{Proposition}
The smallest full subcategory of $\umod_\Lambda$ which contains the
object $V_1=\bk_p$ and is closed under isomorphisms, direct sums,
direct summands and extensions coincides with $\umod_\Lambda$. Hence: 
\be
\langle V_1\rangle=\umod_\Lambda~~.
\ee
\end{Proposition}

\begin{proof}
Let $\cT=\langle V_1\rangle$ be the smallest subcategory of
$\umod_\Lambda$ which is closed under isomorphisms, direct sums,
direct summands and shifts and such that any distinguished triangle
of $\umod_\Lambda$ for which two objects belong to $\cT$ lies in
$\cT$. 

We first show by induction that the modules $V_i$ with $i=2,\ldots, n-1$ 
belong to $\cT$. Consider the AR triangle \eqref{ARtriangles} for $i=1$:
\be
V_1 \stackrel{\underline{g_0}}{\longrightarrow} V_{2}\stackrel{\underline{f_0}}{\longrightarrow} 
V_1\longrightarrow \Omega(V_0)
\ee
where we used the fact that $V_0=0$. Since $V_1\in \Ob\cT$, we have
$V_2\in \Ob\cT$.  Suppose now that $V_{i-1}$ and $V_i$ are objects of
$\cT$ for some $i\in \{2,\ldots, n-1\}$.  Considering the sequence
\eqref{ARtriangles} for $i$, and using the fact that $V_i$ us an
object of $\cT$, we conclude similarly that $V_{i-1}\oplus V_{i+1}$ is
an object of $\cT$. Thus $V_{i+1}$ is also an object of $\cT$ since
$\cT$ is closed under direct summands.  We conclude by induction that
$V_1,\ldots, V_{n-1}$ belong to $\cT$. This gives the conclusion since
$\cT$ is closed under direct sums and $\umod_\Lambda$ is additively
generated by the objects $V_1,\ldots, V_{n-1}$. \qed
\end{proof}

\subsection{Equivalence between $\umod_\Lambda$ and the category of singularities of $\Lambda$}

Recall that the {\em category of singularities} of $\Lambda$ is the
Verdier quotient:
\be
\rD^b_{\sing} (\Lambda)\eqdef \rD^b(\Lambda)/\Perf(\Lambda)~~,
\ee
where $\Perf(\Lambda)$ is the triangulated subcategory of perfect
complexes. In our case, this category is triangle-equivalent with 
$\umod_\Lambda$, as we explain next. 

Recall that the {\em depth} of a Noetherian $\Lambda$-module $M$ is
defined through:
\be
\depth_\Lambda(M)\eqdef\inf_{i\geq 0} \big\{\dim \Ext^i_\Lambda(\bk,M)>0\big\}~~.
\ee
This quantity satisfies the inequality: 
\be
\depth_\Lambda(M)\leq \kdim(A/\Ann(M))\leq \kdim A~~.
\ee
There is another way to formulate this for local rings.  Let $(R,m)$
be a local ring.  Recall that a sequence $x_1, \dots, x_r \in m$ is
called an $M$-sequence if $x_i$ is a non zero divisor in the quotient
$M/\langle x_1, \dots, x_{i-1}\rangle$ for all $1 \leq i \leq r$.  The
depth of a module over a local ring $(R,m)$ is equal to the length of
a maximal $M$-sequence.  A Noetherian $\Lambda$-module is called {\em
maximal Cohen-Macaulay} (MCM) if
$\depth_\Lambda(M)=\kdim(\Lambda)$. Let $\MCM(\Lambda)$ be the full
subcategory of $\mod_\Lambda$ whose objects are the MCM modules.

\

\begin{Lemma}
\label{lemma:MCM}
Any finitely-generated $\Lambda$-module $M$ is maximal Cohen-Macaulay. 
Thus $\MCM(\Lambda)=\mod_\Lambda$.
\end{Lemma}

\begin{proof}
This is well-known, but we sketch the proof for completeness. 
Since $\Lambda$ is an Artinian local ring, it has Krull dimension zero.
On the other hand, the depth of any finitely-generated
$\Lambda$-module is zero since any element of $m$ is nilpotent and hence a divisor of zero.\qed
\end{proof}

\

\begin{Proposition}
\label{prop:sing}
There exists an equivalence of triangulated categories: 
\be
\rD^b_{\sing}(\Lambda)\simeq\umod_\Lambda~~.
\ee
\end{Proposition}

\begin{proof}
Since $\Lambda$ is Gorenstein, there exists \cite{Buchweitz} a natural
equivalence of triangulated categories $\rD^b_{\sing}(\Lambda)\simeq
\uMCM(\Lambda)$, where $\uMCM(\Lambda)$ is the projective
stabilization of $\MCM(\Lambda)$. The conclusion now follows from
Lemma \ref{lemma:MCM}.  \qed
\end{proof}

\subsection{Localization at $U(\Lambda)$}

\noindent Since $\Lambda$ is a local ring with maximal ideal $\langle p\rangle$,
the multiplicative set $\Lambda\setminus \langle p\rangle$ coincides
with the group of units $U(\Lambda)$.

\

\begin{Proposition}
\label{prop:loc_Lambda}
Localization at the multiplicative set $U(\Lambda)=\Lambda\setminus \langle p\rangle$
of units of $\Lambda$ induces an equivalence of triangulated categories:
\be
\uloc_p:\umod_{\Lambda}\stackrel{\sim}{\rightarrow} \umod_{\Lambda_{(p)}}
\ee
\end{Proposition}

\begin{proof}
Multiplication by any $s\in U(\Lambda)$ gives an
isomorphism of the $\Lambda$-modules $V_i\simeq_R \Lambda_i$ for each
$i\in \{1,\ldots, n\}$. Since $\mod_\Lambda$ is additively generated by
$V_1,\ldots, V_n$, it follows that $s$ acts as an isomorphism on any
finitely-generated $\Lambda$-module.  In particular, the localization
functor $\loc_p$ at the multiplicative set $U(\Lambda)$ is an equivalence of categories between
$\mod_\Lambda$ and $\mod_{\Lambda_{(p)}}$. Since this functor is exact,
  it is an equivalence of exact categories. Since $\mod_\Lambda$ is a
  Frobenius category, it follows that the same is true for
  $\mod_{\Lambda_{(p)}}$ and that $\loc_p$ induces a triangulated
  equivalence $\uloc_p$ between the stable categories $\umod_\Lambda$
  and $\umod_{\Lambda_{(p)}}$. \qed
\end{proof}

\

\begin{remark}
We have a natural isomorphism of rings:
\be
\Lambda_{(p)}\simeq R_{(p)}/\langle p^n\rangle~~.
\ee
\end{remark}

\section{Matrix factorizations over an elementary divisor domain}
\label{sec:edd}

\noindent Let $R$ be an elementary divisor domain and $W$ be a non-zero element of $R$. 

\subsection{Isomorphism classes in $\zmf(R,W)$}

The Smith normal form theorem over an elementary divisor domain (see
Appendix \ref{app:edd}) allows us to characterize isomorphism classes
of objects in the category $\zmf(R,W)$.

\

\begin{Proposition}
\label{prop:if}
Let $a=(R^{\rho|\rho},D)$ and $a'=(R^{\rho'|\rho'},D')$ be two finite
rank matrix factorizations of the non-zero element $W\in R^\times$, where
$D=\left[\begin{array}{cc} 0 & v\\ u & 0\end{array}\right]$ and
$D'=\left[\begin{array}{cc} 0 & v'\\ u' & 0\end{array}\right]$. Let
$\bd_1(v),\ldots, \bd_\rho(v)$ and $\bd_1(v'),\ldots, \bd_{\rho'}(v')$
be respectively the invariant factors of the matrices $v\in
\Mat(\rho,\rho,R)$ and $v'\in \Mat(\rho',\rho',R)$. Then the following
statements are equivalent:
\begin{enumerate}[(a)] 
\itemsep 0.0em
\item $a$ and $a'$ are isomorphic in the category $\zmf(R,W)$.
\item We have $\rho=\rho'$ and the invariant factors of $v$ and $v'$
are equal:
\be
\bd_i(v)=\bd_i(v')~,~~\forall i\in \{1,\ldots, \rho\}~~.
\ee
\end{enumerate}
\end{Proposition}

\begin{proof}
By \cite[Proposition 1.4]{elbezout}, the matrix factorizations $a$
and $a'$ are strongly isomorphic iff $\rho=\rho'$ and the matrices $v$
and $v'$ are equivalent. Recall that $u,v,u',v'$ have maximal rank
since $W\neq 0$.  Since $R$ is an EDD, Proposition \ref{prop:equivEDD}
shows that $v$ and $v'$ are equivalent iff $\rho=\rho'$ and their
invariant factors satisfy $\bd_i(v)=\bd_i(v')$ for all $i\in \{1,\ldots,
\rho\}$. \qed
\end{proof}

\noindent The following result shows that any matrix factorization of
$W$ is naturally isomorphic in $\zmf(R,W)$ to a direct sum of
elementary factorizations.

\

\begin{Theorem}
\label{thm:strong_iso}
There exists an autoequivalence $F$ of the category $\zmf(R,W)$
such that: 
\begin{enumerate}
\itemsep 0.0em
\item  $F$ is isomorphic with the identity functor $\id_{\zmf(R,W)}$.
\item For any matrix factorization $a=(R^{\rho|\rho},D)$ of $W$ with 
$D=\left[\begin{array}{cc} 0 & v\\ u & 0\end{array}\right]$, we have: 
\be
F(a)=e_{d_1(v)}\oplus \ldots \oplus e_{d_\rho(v)}~~,
\ee
where $d_1(v), \ldots, d_\rho(v)\in R$ are representatives for the invariant factors of $v$, 
i.e. $d_i(v)\in \bd_i(v)$ for all $i\in \{1,\ldots, \rho\}$.  
\end{enumerate}
\end{Theorem}

\begin{proof}
For any $v\in \Mat(\rho,\rho,R)$, choose invertible matrices $A_v,B_v\in
\GL(\rho,R)$ such that $v_0\eqdef A_v v B_v^{-1}$ is in Smith normal
form:
\be
v_0=\diag\big(d_1(v),\ldots, d_k(v),0,\ldots, 0\big)~~,
\ee  
where $d_k(v)\in \bd_k(v)$. For any matrix
factorization $a=(R^{\rho|\rho},D)$ of $W$ with
$D=\left[\begin{array}{cc} 0 & v\\ u & 0\end{array}\right]$, let
$U_D=\left[\begin{array}{cc} A_v & 0\\ 0 & B_v \end{array}\right]$
and:
\ben
\label{D0def}
D_0\eqdef U_D D U_D^{-1}=\left[\begin{array}{cc} 0 & A_v vB_v^{-1}\\ B_v uA_v ^{-1} & 0\end{array}\right]=\left[\begin{array}{cc} 0 & v_0\\ u_0 & 0\end{array}\right]~~,
\een
where $u_0=B_v u_v A^{-1}$. Since $uv=vu=W$, we have
$u_0v_0=v_0u_0=W$. This requires that $u_0$ is diagonal (we have
$u_0=W v_0^{-1}$ in the field of fractions of $R$), namely we have
$u_0=\diag(u_1,\ldots, u_\rho)$, where $u_i=\frac{W}{d_i(v)}$. Since
$d_i(v)|d_{i+1}(v)$, we have $u_i|u_{i-1}$ and hence $u_0$ is the
reverse Smith normal form of $u$:
\be
u_0=\diag(d_\rho(u_0),\ldots, d_1(u_0))~~.
\ee  
Define $F_0:\Ob\MF(R,W)\rightarrow \Ob\MF(R,W)$ through:
\ben
\label{F0def}
F_0(R^{\rho|\rho},D):=a_0\eqdef (R^{\rho|\rho},D_0)~~.
\een
Notice that $a_0$ coincides with the following direct sum of elementary matrix factorizations: 
\be
a_0=e_{d_1(v)}\oplus \ldots \oplus e_{d_k(v)}~~.
\ee
Moreover, relation \eqref{D0def} implies $D_0 U_D=U_D D$, showing that
$U_D$ is an isomorphism from $a$ to $a_0$ in $\zmf(R,W)$:
\ben
\label{UDiso}
U_D:a\stackrel{\sim}{\rightarrow} a_0~~.
\een
For any morphism $f:a\rightarrow a'$ in $\zmf(R,W)$ with
$a=(R^{\rho|\rho},D)$, $a'=(R^{\rho'|\rho'},D')$ and
$D=\left[\begin{array}{cc} 0 & v\\ u & 0\end{array}\right]$,
$D'=\left[\begin{array}{cc} 0 & v'\\ u' & 0\end{array}\right]$, define
a morphism $F_1(f):a_0\rightarrow a'_0$ in $\zmf(R,W)$ as follows.
Since $f$ is a morphism in $\zmf(R,W)$ it satisfies the condition
$D'f=fD$.  Define:
\ben
\label{F1def}
F_1(f):=f_0\eqdef U_{D'} f U_D^{-1}~~.
\een
where $U_{D}=\left[\begin{array}{cc} A & 0\\ 0 & B \end{array}\right]$
and $U_{D'}\eqdef U_{v'}=\left[\begin{array}{cc} A' & 0\\ 0 &
B' \end{array}\right]$, with $A=A_v$, $B=B_v$, $A'=A_{v'}$ and
$B'=B_{v'}$. Since $D_0=U_D D U_D^{-1}$ and $D'_0=U_{D'}DU_{D'}^{-1}$,
the relation $D'f=fD$ implies $D'_0f_0=f_0D_0$, showing that $f_0$ is
a morphism from $a_0$ to $a'_0$ in $\zmf(R,W)$. If $f$ is the identity
endomorphism, then $f_0$ is the identity endomorphism. If
$g:a'\rightarrow a''$ is another morphism in $\zmf(R,W)$, then we have
$(gf)_0=U_{D''} gf U_{D}^{-1}=U_{D''}gU_{D'}^{-1}
U_{D'}fU_{D}^{-1}=g_0f_0$. This shows that $F=(F_0,F_1)$ is an
endofunctor of $\zmf(R,W)$. Relation \eqref{F1def} shows that the
isomorphisms \eqref{UDiso} satisfy $U_D' f =F_1(f) U_D$ and hence give
an isomorphism of functors:
\be
U:\id_{\zmf(R,W)}\stackrel{\sim}{\rightarrow} F~~.
\ee
In particular, $F$ is an autoequivalence of $\zmf(R,W)$. 
\qed
\end{proof}

\

\noindent The decomposition of a matrix factorization into elementary factorizations is generally non-unique. 
The ambiguity in this decomposition can be characterized as follows.

\

\begin{Corollary}
\label{cor:can}
The following statements hold: 
\begin{enumerate}
\itemsep 0.0em
\item If $e_{v_1},\ldots, e_{v_n}$ are elementary factorizations of $W$, then we have: 
\ben
\label{sum}
e_{v_1}\oplus \ldots \oplus e_{v_n}\simeq_{\zmf(R,W)} e_{d_1}\oplus \ldots \oplus e_{d_n}~~,
\een
where: 
\be
d_k\in \frac{\bdelta_k}{\bdelta_{k-1}}~~\forall k\in \{1,\ldots, n\}~~,
\ee 
with: 
\ben
\label{bdeltak}
\bdelta_k\eqdef \left(\{v_{i_1}\ldots v_{i_k}\,|\,1\leq i_1<\ldots < i_k\leq n\}\right)~~.
\een
Moreover, if $v_1|\ldots |v_n$ then we can take $d_k=v_k$ for all
$k\in \{1,\ldots, n\}$ while if $v_1,\ldots, v_n$ are mutually coprime then we
have $e_{v_1}\oplus \ldots \oplus e_{v_n}\simeq_{\zmf(R,W)}e_1^{\oplus
  n-1}\oplus e_{v_1\ldots v_n}$.
\item If a matrix factorization $a=(R^{\rho|\rho},D)$ of $W$ with
  $D=\left[\begin{array}{cc} 0 & v\\ u & 0\end{array}\right]$
  satisfies:
\be
a\simeq_{\zmf(R,W)}e_{v_1}\oplus \ldots \oplus e_{v_n}
\ee
for some elementary factorizations $e_{v_i}$ such that $v_1|\ldots
|v_n$, then we have $n=\rho$ and $v_i\in \bd_i(v)$ for all $i\in \{1,\ldots,
n\}$. In particular, the strong isomorphism classes of matrix factorizations 
of $W$ are in bijection with finite ascending sequences of principal 
ideals $I_n\subset \ldots \subset I_1$ such that $W\in I_n$. 
\end{enumerate}
\end{Corollary}

\begin{proof}

\

\begin{enumerate} 
\itemsep 0.0em
\item Let $a\eqdef e_{v_1}\oplus \ldots \oplus e_{v_n}$. Then
$a=(R^{\rho|\rho},D)$ with $D=\left[\begin{array}{cc} 0 & v\\ u &
0\end{array}\right]$, where $v=\diag(v_1,\ldots, v_n)$. Since all
non-principal minors of a diagonal matrix vanish, the determinantal
invariants of $v$ coincide with $\bdelta_k$, while the invariant
factors coincide with $\bd_k$. The first statement now follows from
Proposition \ref{prop:equivEDD}.  If $v_1|\ldots |v_n$, then we have
$\bdelta_n=(v_1\ldots v_n)$ and $\bd_n=(v_n)$. If $v_1,\ldots, v_n$
are coprime then we have $\bdelta_1=\ldots = \bdelta_{n-1}=(1)$ and
$\bdelta_n=(v_1\ldots v_n)$, thus $\bd_1=\ldots=\bd_{n-1}=(1)$ and
$\bd_n=(v_1\ldots v_n)$.
\item Follows immediately from Theorem \ref{thm:strong_iso} and point
1. above. \qed
\end{enumerate}
\end{proof}

\

\begin{remark} The critical ideal $\fI_W$ defined in \eqref{fIW}
annihilates the module $\Hom_{\hmf(R,W)}(e_1,e_2)$ for any two
elementary matrix factorizations $e_1$ and $e_2$ of $W$ (see
\cite[Remark 2.2.]{elbezout}).  Using this fact, Corollary
\ref{cor:can} implies $\fI_W\Hom_{\hmf(R,W)}(a,b)=0$ for any two
finite rank matrix factorizations $a,b$ of $W$ (notice that an
isomorphism in $\zmf(R,W)$ induces an isomorphism in $\hmf(R,W)$). In
particular, $\hmf(R,W)$ can be viewed as an $R/\fI_W$-linear
category. Since $W\in \fI_W$, we have a natural epimorphism $R/\langle
W\rangle \rightarrow R/\fI_W$. Thus $\hmf(R,W)$ is in particular an
$R/\langle W\rangle$-linear category.
\end{remark}

\

\begin{remark}
Let $v_1$ and $v_2$ be two divisors of $W$. Then
$\bdelta_1(v)=(v_1,v_2)$ and $\bdelta_2(v)=(v_1v_2)$ and the quantities \eqref{bdeltak} are
$\bd_1=(v_1,v_2)$ and $\bd_2=[v_1,v_2]$. Thus \eqref{sum} takes the form:
\ben
\label{2sum}
e_{v_1}\oplus e_{v_2}\simeq_{\zmf(R,W)} e_{d_1}\oplus e_{d_2}~~
\een
with $d_1\in (v_1,v_2)$ and $d_2\in [v_1,v_2]$. 
If $v_1|v_2$ and $u_i=W/v_i$, then we have $u_2|u_1$ and: 
\be
\Sigma(e_{v_1}\oplus e_{v_2})\simeq_{\zmf(R,W)} e_{u_2}\oplus e_{u_1}~~,
\ee
since $\Sigma e_{v_i}=e_{-u_i}\simeq_{\zmf(R,W)} e_{u_i}$ (see \cite[Section 1.7]{elbezout}).
Corollary \ref{cor:can} shows that the subcategory $\zef(R,W)$ generates
$\zmf(R,W)$ under direct sums with the relations \eqref{2sum}.
At the level of isomorphism classes, these relations
correspond to the operation $(I_1,I_2)\rightarrow (I_1\cap I_2, I_1+I_2)$
on principal ideals $I_1,I_2$ which contain $W$, where the RHS is a chain 
$I_1\cap I_2\subset I_1+I_2$ of principal ideals containing $W$.
\end{remark}

\subsection{Direct sum decompositions in $\hmf(R,W)$}

The results of the previous subsection imply that elementary matrix factorizations 
generate the category $\hmf(R,W)$ under direct sums.

\ 

\begin{Proposition}
\label{prop:strong_iso}
There exists an autoequivalence $\Psi$ of $\hmf(R,W)$ such that:
\begin{enumerate}
\itemsep 0.0em
\item  $\Psi$ is isomorphic with the identity functor $\id_{\hmf(R,W)}$.
\item For any matrix factorization $a=(R^{\rho|\rho},D)$ of $W$ with 
$D=\left[\begin{array}{cc} 0 & v\\ u & 0\end{array}\right]$, we have: 
\be
\Psi(a)=e_{d_1(v)}\oplus \ldots \oplus e_{d_\rho(v)}~~,
\ee
where $d_1(v),\ldots, d_\rho(v)$ are representatives for the invariant factors of $v$. 
\end{enumerate}
In particular, the subcategory $\hef(R,W)$ generates $\hmf(R,W)$ under direct
sum. Thus any matrix factorization $a\in \Ob(\MF(R,W))$ is isomorphic
in $\hmf(R,W)$ with a direct sum of a finite collection of elementary
factorizations.
\end{Proposition}

\begin{proof} Follows immediately from Theorem \ref{thm:strong_iso} upon taking 
$\Psi$ to be the autoequivalence of $\hmf(R,W)$ induced
by the autoequivalence $F$ of $\zmf(R,W)$. \qed
\end{proof}

\

\noindent Notice that the decomposition of an object of $\hmf(R,W)$ as
a finite direct sum of elementary factorizations need not be unique up
to permutation and isomorphisms in $\hmf(R,W)$. Moreover, an
elementary factorization need not be an indecomposable object of
$\hmf(R,W)$.

\

\begin{remark}
For any B\'ezout domain $R$, let $\mhef(R,W)$ be the subcategory of
$\hmf(R,W)$ which is additively generated by elementary
factorizations. In \cite[Conjecture 3.4]{elbezout} it was conjectured
that the inclusion functor:
\be
\iota: \mhef(R,W) \to \hmf(R,W)
\ee
is an equivalence of $R$-linear categories when $W$ is a
critically-finite element. Proposition \ref{prop:strong_iso} proves
this conjecture when $R$ is an elementary divisor domain, under the
weaker hypothesis that $W$ is any non-zero element of $R$. It is an
open question whether all B\'ezout domains are elementary divisor
domains.
\end{remark}

\subsection{Cones over morphisms between elementary factorizations}

Let $e_{v_1}$ and $e_{v_2}$ be elementary matrix factorizations of $W$
and set $u_i\eqdef W/v_i$. By \cite[Proposition 2.2]{elbezout},
morphisms $f:e_{v_1}\rightarrow e_{v_2}$ in $\hmf(R,W)$ have the form
$f=r \cdot \left[\begin{array}{cc} \frac{v_2}{d} & 0 \\ 0 &
\frac{v_1}{d}
\end{array}\right]$, where $r$ is an arbitrary element of $R$ and
$d\in (v_1,v_2)$ is a gcd of $v_1$ and $v_2$.

\

\begin{Proposition}
\label{prop:conedec}
Let $f:e_{v_1}\rightarrow e_{v_2}$ be a morphism in $\hmf(R,W)$ corresponding to the element $r\in R$. Let:
\ben
\label{xizeta}
\xi \in \frac{(v_1,v_2,u_1,u_2, r)(v_1)}{(v_1,v_2)} ~~\mathrm{and}~~\zeta \eqdef -\frac{v_1u_2}{\xi}~~.
\een
Then there exists an isomorphism in $\zmf(R,W)$:
\ben
\label{conedec}
C(f)\simeq_{\zmf(R,W)} e_\xi\oplus e_\zeta~~.
\een
\end{Proposition}

\begin{proof}
Let $d\in (v_1,v_2)$ be a gcd of $v_1$ and $v_2$. Using Definition
\ref{def:cone}, we find that the mapping cone of $f$ is given by:
\be
C(f)=\left[\begin{array}{cccc}
0 & 0 & -u_1 & 0 \\
0 & 0 & r \cdot \frac{v_2}{d} & v_2 \\
-v_1 & 0 & 0 & 0 \\
r \cdot \frac{v_1}{d} & u_2 &  0 & 0 \\
\end{array}\right]~~.
\ee
Since $R$ is an elementary divisor domain, the matrices
$A\eqdef\left[\begin{array}{cc} -v_1 & 0 \\ r \cdot \frac{v_1}{d} & u_2
\end{array}\right]$ and $B\eqdef\left[\begin{array}{cc} -u_1 & 0 \\ r \cdot
\frac{v_2}{d} & v_2
\end{array}\right]$ can be reduced to Smith normal form (see Appendix
\ref{app:edd}). Furthermore, since $AB=W$ we can find invertible
matrices $P$ and $Q$ such that $PAQ$ and $QBP$ have normal forms.
Let $\xi\in \left( v_1, u_2, r \cdot \frac{v_1}{d}\right)$.
Then $PAQ=\left[\begin{array}{cc} \alpha_1 & 0 \\ 0 & \alpha_2
\end{array}\right]$, where $\alpha_i$ are invariant factors of $A$. By
definition, $\alpha_1$ is a greatest common divisor of all entries of
$A$, which we can take to equal $\xi$. On the other hand, we have
$\alpha_2=\frac{\det A}{\xi}=-\frac{v_1 u_2}{\xi}=\zeta$.  Hence the
Smith normal form of $B$ equals $\left[\begin{array}{cc} \frac{W}{\xi}
& 0 \\ 0 & \frac{W}{\zeta}
\end{array}\right]$. We conclude that $C(f)$ is isomorphic in
$\zmf(R,W)$ with the matrix:
\be
C_0(f)=\left[\begin{array}{cccc}
0 & 0 & \frac{W}{\xi} & 0 \\
0 & 0 & 0 & \frac{W}{\zeta} \\
\xi & 0 & 0 & 0 \\
0 & \zeta &  0 & 0 \\
\end{array}\right]=e_\xi\oplus e_\zeta~~.
\ee 
Let $s\eqdef (v_1,v_2,u_1,u_2)\in R/U(R)$ and $b\eqdef \frac{(v_1)}{(v_1,v_2)}\in R/U(R)$. 
By \cite[eqs. (2.4)]{elbezout}, we have $(v_1,u_2)=(s)(b)$. Thus:
\ben
\left(v_1, u_2, r \cdot \frac{v_1}{(v_1,v_2)}\right)=(s b, r b)=(s,r)b~~,
\een
which shows that \eqref{xizeta} holds. \qed
\end{proof}

\

\begin{Corollary}
Let $f: e_{v_1} \to e_{v_2}$ be a morphism in $\hmf(R,W)$ which
corresponds to an element $r\in R$ and let $\xi$ and $\zeta$ be as in
Proposition \ref{prop:conedec}.  Then $f$ is an isomorphism in
$\hmf(R,W)$ if and only if the following relations hold in $R/U(R)$:
\be
(\xi,W/\xi)=(\zeta,W/\zeta)=(1).
\ee
\end{Corollary}

\begin{proof}
The morphism $f$ is an isomorphism in the additive triangulated category
  $\hmf(R,W)$ iff $C(f)$ is a zero object. By Proposition
  \ref{prop:conedec}, this happens iff both $e_\xi$ and $e_\zeta$ are
  zero objects. By \cite[Corollary 2.11]{elbezout}, this is the case iff
  $(\xi,W/\xi)=(\zeta,W/\zeta)=(1)$. \qed
\end{proof}

\subsection{Primary matrix factorizations}

Recall that an element of $R$ is called primary if it is a power of a prime element. 

\

\begin{Definition}
An elementary factorization $e_v$ of $W$ is called {\em primary} if $v$ is a primary divisor of $W$. 
\end{Definition}

\

\noindent Let $e_v$ be a primary matrix factorization of $W$. Then $v=p^i$ for
some prime divisor $p$ of $W$ and some integer $i\in \{0,\ldots, n\}$,
where $n$ is the order of $p$ as a divisor of $W$.  We have $W=p^n
W_1$ for some element $W_1\in R$ such that $p$ does not divide $W_1$
and $u=p^{n-i}W_1$. Thus $(u,v)=(p^{\min(i,n-i)})$.

\

\begin{Definition}
The prime divisor $p$ of $W$ is called the {\em prime support} of
$e_v$. The order $n$ of $p$ is called the {\em order} of $e_v$ while
the integer $i\in \{0,\ldots, n\}$ is called the {\em size} of $e_v$.
\end{Definition}

\subsection{A Krull-Schmidt theorem for $\hmf(R,W)$ when $W$ is critically-finite}

Recall that an object of an additive category is called {\em
  indecomposable} if it is not isomorphic with a direct sum of two
non-zero objects. A {\em Krull-Schmidt category} is an additive
category for which every object decomposes into a finite direct
sum of objects having quasi-local endomorphism rings.

\

\begin{Theorem}
\label{thm:Krull}
Let $W$ be a critically-finite element of $R$. Then $\hmf(R,W)$ is a
Krull-Schmidt category whose non-zero indecomposables are the nontrivial
primary matrix factorizations of $W$. In particular, $\hmf(R,W)$ is
additively generated by $\hef_0(R,W)$.
\end{Theorem}

\begin{proof} By \cite[Proposition 3.1]{elbezout} and \cite[Theorem
3.2]{elbezout}, any elementary matrix factorization decomposes into a
finite direct sum of primary matrix factorizations. On the other hand,
any matrix factorization of $W$ decomposes as a finite direct sum of
elementary factorizations and hence also as a finite direct sum of
primary factorizations whose prime supports are the prime divisors of
$W$. By \cite[Proposition 2.24]{elbezout}, every primary matrix
factorization has a quasi-local endomorphism ring.  \qed
\end{proof}

\

\begin{Corollary}
\label{cor:krull} Let $W\in R$ be an element of $R$ which has a finite
prime decomposition. Then $\hmf(R,W)$ is a Krull-Schmidt category
whose indecomposables are the nontrivial primary matrix factorizations
of $W$.
\end{Corollary}

\begin{proof} Write $W=W_0p_1^{n_1}\ldots p_N^{n_N}$, where $p_j$ are
the critical prime divisors of $W$, $n_j\geq 2$ and $W_0$ is the
product of the non-critical prime divisors of $W$. Then $W_0$ is
non-critical and we can apply Theorem \ref{thm:Krull}. \qed
\end{proof}

\

\begin{remark} Theorem \ref{thm:Krull} proves \cite[Conjecture
3.5]{elbezout} when $R$ is an elementary divisor domain.
\end{remark}

\subsection{The category $\hmf_p(R,W)$ and its equivalent descriptions}

Let $p$ be a prime divisor of $W$ of order $n$. Let $\hmf_p(R,W)$
denote the smallest strictly full\footnote{I.e., full and closed under
  isomorphisms.}  subcategory of $\hmf(R,W)$ which is closed under
direct sums and contains all those primary factorizations of $W$ which
have prime support $p$. Propositions \ref{prop:strong_iso} and \cite[Proposition 3.1]{elbezout}
imply that $\hef(R,W)$ is additively generated by its strictly
full subcategory $\hef_0(R,W)$ whose objects are the primary
factorizations of $W$.

\

\begin{Lemma}
\label{lemma:HomCriterion}
A matrix factorization $a$ of $W$ is an object of $\hmf_p(R,W)$ iff
$\Hom(e_q, a)=0$ for any prime divisor $q$ of $W$ such that $(q)\neq (p)$.
\end{Lemma}

\begin{proof}
Since $\hmf(R,W)$ is additively generated by $\hef_0(R,W)$, it
suffices to prove the statement when $a=e_v$ is a primary matrix
factorization. In this case, we have $v=s^k$ for some prime divisor
$s$ of $W$ and:
\be
\Hom_{\hmf(R,W)}(e_q, a)=\Hom_\hmf(R,W)(e_q, e_{s^k})\simeq R/\langle q,s^k\rangle\simeq_R \twopartdef{R/(s)}{(q)=(s)}{0}{(q)\neq (s)}~~.
\ee 
Hence $\Hom(e_q, a)$ vanishes for any prime divisor $q$ of $W$ such that
$(q)\neq (p)$ iff $(s)=(p)$, which is equivalent with the condition
that $e_v$ is an object of $\hmf_p(R,W)$. \qed
\end{proof}

\

\begin{Proposition}
$\hmf_p(R,W)$ is a triangulated subcategory of $\hmf(R,W)$.
\end{Proposition}

\begin{proof}
The subcategory $\hmf_p(R,W)$ of $\hmf(R,W)$ is strictly full by
definition. Since $\hmf_p(R,W)$ is additively generated by primary
factorizations of prime support $p$, \cite[Proposition 2.26]{elbezout} implies that $\hmf_p(R,W)$ is closed under
suspension.  Let $a\rightarrow b\rightarrow c\rightarrow \Sigma a$ be
a distinguished triangle of $\hmf(R,W)$ such that $a$ and $b$ are
objects of $\hmf_p(R,W)$. For any prime divisor $q$ of $W$ such that
$q\not \simeq p$, the homological functor $\Hom_{\hmf(R,W)}(e_q,-)$
takes this triangle into a long exact sequence:
\ben
\label{lex}
\ldots \longrightarrow \Hom_{\hmf(R,W)}(e_q,b)\longrightarrow \Hom_{\hmf(R,W)}(e_q,c) \longrightarrow \Hom_{\hmf(R,W)}(e_q,\Sigma a) \longrightarrow \ldots 
\een
Since $b$ and $\Sigma a$ are objects of $\hmf_p(R,W)$, we have
$\Hom_{\hmf(R,W)}(e_q,b)=\Hom_{\hmf(R,W)}(e_q,\Sigma a)=0$ by Lemma
\ref{lemma:HomCriterion} and the sequence \eqref{lex} implies
$\Hom_{\hmf(R,W)}(e_q,c)=0$. Applying Lemma \ref{lemma:HomCriterion}
once again, we conclude that $c$ is an object of $\hmf_p(R,W)$.
Since triangles can be rotated, it follows that any triangle in 
$\hmf(R,W)$ for which two objects are in $\hmf_p(R,W)$ has all its 
objects in $\hmf_p(R,W)$. \qed
\end{proof}

\

\begin{Proposition}
For any prime element $p\in R$, the ring $R_{(p)}$ is discrete valuation ring. 
In particular, we have $\kdim R_{(p)}=1$. 
\end{Proposition}

\begin{proof}
The maximal ideal of  $R_{(p)}$ is the principal ideal $(p)$. 
The powers of this ideal form the strictly descending sequence: 
\be
R_{(p)}\supsetneq \langle p\rangle \supsetneq \langle p^2\rangle \supsetneq \ldots ~~.
\ee
The same argument as in the proof of Proposition \ref{prop:Artinian}
(but with $R$ replaced by $R_{(p)}$) shows that these and the zero
ideal are all the ideals of $R_{(p)}$. In particular, any strictly
ascending sequence of ideals terminates and hence $R_{(p)}$ is
Noetherian and thus a PID.  Moreover, we have\footnote{If $x\in
\cap_{i=1}^\infty (p^i)$, then $(x)\subset (p^i)$ for all $i$, which
requires $x=0$ since otherwise $(x)$ would equal some $(p^j)$.}
$\cap_{i=1}^\infty (p^i)=0$. The zero ideal is prime since $R_{(p)}$
is an integral domain and we have $\langle 0\rangle \neq \langle
p\rangle$. Hence $\kdim R_{(p)}=1$, which implies that $R_{(p)}$ is
not a field. \qed
\end{proof}

\

\begin{remark}
Since any discrete valuation ring is a regular local ring, it follows
that $R_{(p)}$ is a regular local ring.
\end{remark}

\

\begin{Proposition}
\label{prop:locp}
Let $p$ be a prime element of $R$ and $n>0$ be a positive integer.
Then the localization functor $\loc_p:\hmf(R,p^n)\rightarrow
\hmf(R_{(p)},p^n)$ at the multiplicative set $S_p\eqdef R\setminus
\langle p\rangle$ is a triangulated equivalence.
\end{Proposition}

\begin{proof} 
Let $W=p^n$. We have:
\be
S_{p^n}\eqdef\big\{r\in R\, |\, (r,p^n)=(1)\big\}=\big\{r\in R \, | \, (r,p)=(1)\big\}=\big\{r\in R \, |\, p\!\!\not | r\big\}=R\setminus \langle p\rangle=S_p~~.
\ee
Hence \cite[Proposition 2.15]{elbezout} implies that $\loc_p$ is an
$R$-linear equivalence between $\hef(R,p^n)$ and $\hef(R_{(p)},p^n)$.
Since $R$ is B\'ezout (and hence Pr\"ufer), the localization $R_{(p)}$ is a
(possibly non-Noetherian) valuation domain and hence a B\'ezout domain.
Since any local B\'ezout domain is an EDD \cite[Corollary 2.3]{LLS}, it
follows that $R_{(p)}$ is an EDD. Since both $R$ and $R_{(p)}$ are
EDDs, the categories $\hmf(R,p^n)$ and $\hmf(R_{(p)},p^n)$ are
additively generated by $\hef(R,p^n)$ and $\hef(R_{(p)},p^n)$.  Thus
$\loc_p$ is an $R$-linear equivalence between
$\hmf(R,p^n)$ and $\hmf(R_{(p)},p^n)$. This implies the conclusion
since $\loc_p$ is a triangulated functor by \cite[Proposition
2.12]{elbezout}. \qed
\end{proof}

\

\begin{Proposition}
\label{prop:hmfp}
Let $p$ be a prime divisor of $W$ of order $n$. Then the categories
$\hmf_p(R,W)$ and $\hmf_p(R_{(p)},p^n)$ are triangle-equivalent.
\end{Proposition}

\begin{proof}
By \cite[Proposition 2.12]{elbezout}, localization at the
multiplicative set $S_p=R\setminus \langle p\rangle$ gives a
triangulated functor $\loc_p:\hmf(R,W)\rightarrow \hmf(R_{(p)}, W_p)$,
which restricts to a triangulated functor:
\be
\label{locres}
\loc_p:\hmf_p(R,W)\rightarrow \hmf(R_{(p)},
W_p)~~.  
\ee
This restricts to a functor $\Phi:\hef_p(R,W)\rightarrow
\hef(R_{(p)},W_p)$ which maps the elementary factorization $e_{p^i}$
of $W$ to the elementary factorization $e'_{p^i}$ of $W_p$. It is
clear that the functor $\Phi$ is essentially surjective. It is also
fully faithful, since any element $s\in S_p=R\setminus \langle
p\rangle$ acts as an automorphism of each module
$\Hom_{\hmf(R,W)}(e_{p^i}, e_{p^j})\simeq R/\langle
p^{\min(i,j)}\rangle$ by \cite[Lemma 2.14]{elbezout}. Since
$\hef(R,W)$ and $\hef(R_{(p)}, W_p)$ additively generate $\hmf(R,W)$
and $\hmf(R_{(p)}, W_p)$, we conclude that \eqref{locres} is a
triangulated equivalence. On the other hand, the localization $W_p$ of
$W$ at is associated in the ring $R_{(p)}$ with the element $p^n\in
R_{(p)}$. This gives a triangulated equivalence $\hmf(R_{(p)},
W_p)\simeq \hmf(R_{(p)}, p^n)$ by Proposition
\ref{prop:uniteq}. Composing this with \eqref{locres} gives the
conclusion. \qed
\end{proof}

\

\noindent Composing the triangulated equivalences of Propositions
\ref{prop:locp} and \ref{prop:hmfp} gives a triangulated equivalence
$\hmf_p(R,W)\simeq \hmf(R,p^n)$. We have a commutative diagram of
triangulated categories and triangulated equivalences:
\be
\xymatrix{
~~~~\ar[dd]_{\loc_p} \hmf_p(R,W) \ar[rr] & \!  & \hmf(R,p^n) \ar[dd]^{\loc_p}~~~~\\
& & \\
\hmf(R_{(p)},W_p)  \ar[rr] & & \hmf(R_{(p)},p^n) &
}
\ee

\

\begin{Proposition}
\label{prop:eis}
The restriction to $\hmf_p(R,p^n)$ of the cokernel functor of $\hmf(R,p^n)$:
\ben
\label{Cok}
\Cok:\hmf_p(R,p^n)\rightarrow \umod(R/\langle p^n\rangle)=\umod_{A_n(p)}
\een
is a triangulated equivalence.
\end{Proposition}

\begin{proof}
Since $R_{(p)}$ is a local ring, the Eisenbud correspondence
\cite{Eisenbud} gives a triangulated equivalence:
\be
\cok:\hmf(R_{(p)},p^n)\stackrel{\sim}{\rightarrow} \umod_{R_{(p)}/\langle p^n\rangle}~~,
\ee
where $\cok$ is the cokernel functor of $\hmf(R_{(p)},p^n)$.  
By Proposition \ref{prop:locp}, localization at the multiplicative set 
$R\setminus \langle p\rangle$ gives a triangulated equivalence: 
\be
\loc_p:\hmf(R,p^n)\stackrel{\sim}{\rightarrow} \hmf(R_{(p)},p^n)~~.
\ee
By Proposition \ref{prop:loc_Lambda}, localization at the
multiplicative set $U(R/\langle p^n\rangle )$ gives a triangulated equivalence:
\be
\uloc_p: \umod_{R_{(p)}/\langle p^n\rangle }\stackrel{\sim}{\rightarrow} \umod_{R/\langle p^n\rangle}~~.
\ee
It is easy to see that the we have the relation: 
\be
\uloc_p\circ \Cok=\cok\circ \loc_p~~,
\ee
which implies that $\Cok=\uloc_p^{-1}\circ \cok\circ \loc_p$ is a triangulated 
equivalence. \qed 
\end{proof}

\paragraph{Explicit description of $\hmf(R,p^n)$}
Let $p\in R$ be a prime element and $n\geq 2$.  By Theorem
\ref{thm:Krull}, the indecomposable objects of the Krull-Schmidt
category $\hmf(R,p^n)$ are the non-zero primary factorizations of the
critically-finite element $W=p^n$. For any $1\leq i\leq k-1$, let:
\ben
\label{ei}
e_i:=e_{v_i}=\left[\begin{array}{cc}0 & p^i\\ p^{n-i} &
  0\end{array}\right]
\een
be the non-zero primary matrix factorization of $W=p^n$ corresponding
to the primary divisor $v_i=p^i$. For this factorization, we have
$u_i=p^{n-i}$ and $(u_i,v_i)=(p^{\min(i,n-i)})\neq (1)$. Notice
that $e_i$ has order $\delta_n(i)$, where $\delta_n(i)$ was defined in
\eqref{deltamu}. For any $i,j\in \{1,\ldots, n-1\}$, we have
$(v_1,v_2,u_1,u_2)=(p^{\mu(i,j)})$, where $\mu(i,j)$ was defined in
\eqref{deltamu}. Thus \cite[Proposition 2.2]{elbezout} shows that:
\be
\Hom_{\hmf(R,p^n)}(e_i,e_j)\simeq_R R/\langle p^{\mu(i,j)}\rangle ~~
\ee
is a cyclically-presented cyclic module generated by the morphism:
\be
\epsilon_\0(v_i,v_j)\eqdef \left[\begin{array}{cc} p^{j-\min(i,j)} & 0\\ 0 & p^{i-\min(i,j)}\end{array}\right]~~.
\ee
On the other hand, \cite[Proposition 2.8]{elbezout} shows that the composition of morphisms is given by: 
\be
f\circ g=p^{\rho(i,j,k)}rs\epsilon_\0(v_i,v_k) ~~~~~,
\ee
for all $f=r\epsilon_\0(v_j,v_k) \in \Hom_{\hmf(R,p^n)}(e_j,e_k)$ and $g=s\epsilon_\0(v_i,v_j)\in \Hom_{\hmf(R,p^n)}(e_i,e_j)$, where $r,s\in R$ and: 
\be
\rho(i,j,n)=\max(i,j,n)-\min(i,j,n)+\min(i,n)-\max(i,n)=\max(i,j,n)-\min(i,j,n)-|i-n|~~.
\ee
Since $p^n\in \Ann(\Hom_{\hmf(R,p^n)}(e_i,e_j))$, we can view $\hmf(R,p^n)$ as an
$A_n(p)$-linear category. The triangulated equivalence \eqref{Cok}
sends the primary matrix factorization $e_{v_i}$ to the cyclic
$A_n(p)$-module $\Cok(v_i)=V_i$. For any $i,j\in \{1,\ldots, n-1\}$, we have:
\be
\Hom_{\hmf(R,p^n)}(e_i,e_j)\simeq R/\langle p^{\mu(i,j)}\rangle \simeq \uHom_{A_n(p)}(V_i,V_j)~~,
\ee
where the last isomorphism follows from Proposition \ref{prop:umod}.

\subsection{Proof of the main theorem}

\

\

\begin{Proposition}
\label{prop:hmfdec}
Let $W$ be a critically-finite element of $R$ with decomposition
\eqref{Wcritform}.  Then we have an orthogonal decomposition:
\be
\hmf(R,W)=\vee_{i=1}^N \hmf_{p_i}(R,W)~~,
\ee
where $\vee$ denotes the orthogonal sum of triangulated categories.
\end{Proposition}

\begin{proof} Theorem \ref{thm:Krull} and \cite[Proposition
3.1]{elbezout} imply that $\hmf(R,W)$ is additively generated (and
hence also triangle-generated) by the triangulated subcategories
$\hmf_{p_i}(R,W)$. These categories are mutually orthogonal by
\cite[Lemma 2.25]{elbezout}.  \qed
\end{proof}

\

\noindent We are now ready to prove Theorem \ref{thm:main}. 

\

\begin{proof}[of Theorem \ref{thm:main}]
The first equivalence in \eqref{tequiv} follows from Propositions \ref{prop:hmfdec}
and \ref{prop:eis}. The second equivalence follows from Proposition \ref{prop:sing}. 
The fact that $A_n(p)$ is Artinian follows from Proposition \ref{prop:Artinian}. \qed
\end{proof}

\section{Some examples}

\noindent In this section, we discuss a few classes of examples to which the results of the previous sections apply. 

\subsection{Holomorphic matrix factorizations over a non-compact Riemann surface}

Let $\Sigma$ be any connected, smooth and borderless non-compact
Riemann surface\footnote{Notice that such a Riemann surface $\Sigma$
need not be algebraic. In particular, $\Sigma$ may have infinite genus
as well as an infinite number of Freudenthal ends.}. Then $\Sigma$ is
Stein by a result of \cite{BS}.  Moreover, any holomorphic vector
bundle defined on $\Sigma$ is holomorphically trivial (see
\cite[Theorem 30.3]{FO}), so in particular $\Sigma$ has trivial
canonical line bundle. The critical set $Z_W$ of any non-constant
holomorphic function $W:\Sigma\rightarrow \C$ consists of isolated
points, so the total cohomology category $\HF(\Sigma,W)$ of
holomorphic factorizations of $W$ defined in \cite{lg2} can be
identified with the total cohomology category $\HMF(\O(\Sigma),W)$ of
finite rank matrix factorizations of $W$ over the ring $\O(\Sigma)$ of
holomorphic complex-valued functions defined on $\Sigma$ (see
\cite[Proposition 7.1]{lg2}). In particular, the even subcategory
$\HF^\0(\Sigma,W)$ can be identified with the homotopy category of
matrix factorizations $\hmf(\O(\Sigma),W)$. When the set $Z_W$ is
finite, the category $\HF(\Sigma,W)\simeq \HMF(\O(\Sigma),W)$ coincides
with the category of D-branes of a B-type open-closed topological
Landau-Ginzburg model with finite-dimensional on-shell state spaces
(see \cite{LG1,LG2,lg1}).

The non-Noetherian ring $\O(\Sigma)$ is an elementary divisor domain
\cite{Helmer2,Henriksen3,Alling1,Alling2} whose prime elements are
those holomorphic functions having a single simple zero and no other
zeros. For each point $z\in \Sigma$, we thus have a prime element
$p_z\in \O(\Sigma)$ (a holomorphic function which has a simple zero at
$z$ and no other zeroes) which is determined by $z$ up to
multiplication with a non-zero complex constant.  A critically-finite
superpotential is a holomorphic function $W\in \O(\Sigma)$ of the form
$W=W_0W_c$, where $W_0\in \O(\Sigma)$ has only simple zeros (the
number of which may be countably infinite) while $W_c\in \O(\Sigma)$
has a finite number of zeros $z_1,\ldots, z_N\in \Sigma$, each of
which has multiplicity $n_i\geq 2$ and differs from all zeros of
$W_0$. The critical set $Z_W$ of such a holomorphic function contains
the set $\{z_1,\ldots, z_N\}$. In this case, Theorem \ref{thm:main}
shows that the triangulated category $\HF^\0(\Sigma,W)\simeq
\hmf(\O(\Sigma),W)$ is the orthogonal direct sum of the Krull-Schmidt
triangulated categories $\umod_{\O(\Sigma)/(p_{z_i}^{n_i})}$
associated with the points $z_i$, whose Auslander-Reiten quivers are
entirely determined by the multiplicities $n_i$. The Auslander-Reiten
quiver of $\umod_{\O(\Sigma)/(p_{z_i}^{n_i})}$ has $n_i-1$ nodes and
is of the type shown in Figure \ref{fig:squiver}. Notice that only the
critical points $z_1,\ldots, z_N$ ``contribute'' to the orthogonal
decomposition of the category $\hmf(\O(\Sigma),W)$.

\subsection{Valuation domains}

Recall that a unital commutative ring is called a {\em generalized
  valuation ring} \cite{Warfield} if its elements are linearly
preordered by divisibility, i.e. if any two elements $x,y\in R$
satisfy one of the conditions $x|y$ or $y|x$. The following
characterizations are well-known \cite{Brandal,Warfield}:

\

\begin{Proposition}
\label{prop:gvr}
Let $R$ be a unital commutative ring. Then the following statements 
are equivalent:
\begin{enumerate}[(a)]
\itemsep 0.0em
\item $R$ is a generalized valuation ring.
\item The principal ideals of $R$ are linearly ordered by inclusion.
\item The ideals of $R$ are linearly ordered by inclusion.
\item $R$ is quasilocal and any finitely-generated ideal of $R$ is
  principal.
\item If $x_1,\ldots, x_n$ are elements of $R$, then there exists
  $j\in \{1,\ldots, n\}$ such that $\langle x_1,\ldots, x_n\rangle=\langle x_j\rangle$.
\end{enumerate}
In particular $R$, is a generalized valuation ring iff $R$ is a
quasilocal B\'ezout ring.
\end{Proposition}

\

A {\em valuation domain}\footnote{In some references, generalized
valuation domains are called ``valuation rings'', while discrete
valuation domains are called ``discrete valuation rings''.}  is a
generalized valuation ring which is an integral domain. Denote by $K$
the field of fractions of an integral domain $R$. Then $R$ is a
valuation domain iff any $x\in K^\times$ satisfies $x\in R$ or $1/x\in
R$.  An integral domain $R$ is a valuation domain iff there exists a
totally-ordered Abelian group $(G,+,\leq)$ (called the {\em value
group} of $R$) and a surjective valuation $v:K^\times\rightarrow G$
such that $R=\{x\in K^\times|v(x)\geq 0\}\cup \{0\}$. In this case,
$(G,+,\leq)$ is torsion-free \cite{Levi} and order-isomorphic with the
group of divisibility of $R$ (see Subsection \ref{sec:divgroup}). In fact,
a classical result of Krull \cite{Krull} states that any
totally-ordered Abelian group arises as the value group of a valuation
domain.  By Proposition \ref{prop:gvr}, a valuation domain is the same
as a quasilocal B\'ezout domain. Moreover, \cite[Corollary 2.3]{LLS}
shows that a valuation domain is an elementary divisor domain and that
any finitely-presented module over a valuation domain is a direct sum
of cyclic modules. 

\

\begin{Proposition} 
\label{prop:gvrprime}
Let $R$ be a valuation domain. Then $R$ has prime elements iff the
(unique) maximal ideal of $R$ is principal and different from zero. In
this case, any two prime elements of $R$ are associated in
divisibility.
\end{Proposition}

\begin{proof} By Proposition \ref{prop:gvr}, $R$ is a quasilocal
B\'ezout domain. Thus Lemma \ref{lemma:prime} applies, showing that
any prime element $p\in R$ generates a maximal ideal. Since $R$ is
quasilocal, this ideal must coincide with the unique maximal ideal of
$R$, which therefore must be principal and different from zero. By the
same token, any two prime elements of $R$ must generate the same ideal
(namely the maximal ideal of $R$) and hence they must be associated in
divisibility. Conversely, if the maximal ideal of $R$ is principal and
different from zero, then any generator of this ideal is a prime
element of $R$ since maximal ideals are prime ideals. \qed
\end{proof}

\

\begin{Proposition} 
\label{prop:cfgvr}
Let $R$ be a valuation domain with a prime element $p$ and $W \in R$
be a non-zero non-unit of $R$. Then the following statements are
equivalent:
\begin{enumerate}[(a)]
\itemsep 0.0em
\item $W$ is critically-finite.
\item We have $W=u p^n$ for some $n\geq 2$ and some unit $u$ of $R$.
\item We have $W \in \langle p^n\rangle \setminus \langle p^{n+1}\rangle$ for
some $n \geq 2$.
\end{enumerate}
In this case, the category $\hmf(R,W)$ is triangle-equivalent to
$\umod_{R/\langle p^n \rangle}$.
\end{Proposition}

\begin{proof} 
By Proposition \ref{prop:gvrprime}, the ideal $m=\langle
p\rangle$ coincides with the maximal ideal of $R$. Since $R$ is
quasi-local, we have $U(R)=R\setminus m$.
\begin{enumerate}
\itemsep 0.0em
\item $(a)\Rightarrow (b)$. If $W$ is critically-finite, then
$W=W_0W_c$ with $W_c=p^m$ for some $m\geq 2$ and some square-free
element $W_0\in R^\times$. If $W_0$ is a unit, then we can take
$u=W_0$ and $n=m$. If $W_0$ is not a unit, then $W_0\in R\setminus
U(R)=m$ and hence $p$ divides $W_0$. Since $W_0$ is square-free, it
follows that $p$ does not divide $u\eqdef W_0/p$, thus $u$ belongs to
the complement of $m$ and hence is a unit.  In this case, we have $W=u
p^{m+1}$ and we can take $n=m+1$.
\item $(b)\Rightarrow (c)$. If $W=u p^n$ with $u\in U(R)$ and $n\geq
2$, then $W\in \langle p^n\rangle$. Since $U(R)=R\setminus m$, the
prime $p$ cannot divide $u$, hence $W\not\in \langle p^{n+1}
\rangle$. Thus $W\in \langle p^n\rangle \setminus \langle
p^{n+1}\rangle$.
\item $(c)\Rightarrow (a)$. Suppose that $W\in \langle p^n\rangle
\setminus \langle p^{n+1}\rangle$ for some $n\geq 2$. Then $W=u p^n$
for some $u\in R\setminus \{0\}$. Since $W \not\in \langle
p^{n+1}\rangle $, the prime $p$ does not divide $u$ and hence $u\in
R\setminus m=U(R)$ is a unit. In particular, $u$ is square-free and
hence $W$ is critically-finite.
\end{enumerate} The remaining statement follows immediately from
Theorem \ref{thm:main}.  \qed
\end{proof}

\begin{example}
\label{ex:nonNoetheriandvr}
We give several examples of non-Noetherian valuation domains. 
\begin{enumerate}
\item Let $G=\Z^n$ for some $n \geq 2$, totally ordered using the
lexicographic order $\leq_{\mathrm{lex}}$.  Since $G$ is not cyclic,
it is not isomorphic to $\Z$. Hence the valuation domain associated to
$(\Z^n,\leq_{\mathrm{lex}})$ is not Noetherian (see Subsection
\ref{sec:dvr}).  It has exactly one principal prime ideal which is
also maximal. 
Let $e_i$ for $1 \leq i \leq n$ be the canonical basis elements of the
free $\Z$-module $\Z^n$. The inequality $e_i<_{\mathrm{lex}} e_j$ for
$i<j$ implies that the principal filter $\uparrow e_1$ is
prime. However, the filters $\uparrow e_i$ for $i>1$ are
not prime. For details on prime filters see Subsection \ref{sec:divgroup}.

\item Let $K$ be a field and $x$ be an element which is transcendental
over $K$.  For any prime number $p$, consider the tower of integral
domains: \be K[x] \subset K[x^{1/p}] \subset \dots \subset
K[x^{1/p^k}] \subset \dots ~.  \ee For any $k\geq 0$, let $m_k$ be the
maximal ideal of $K[x^{1/p^k}]$ which is generated by the element
$x^{1/p^k}$.  The localization $R_k=K[x^{1/p^k}]_{m_k}$ at the
multiplicative system given by the complement of $m_k$ is a Noetherian
discrete valuation domain. The ring $R\eqdef \cup_{k\geq 0} R_k$ is a
non-Noetherian valuation domain of Krull dimension 1 whose value group
is given by $G=\{\frac{m}{p^k} \, |\, m \in \Z, k \in \N\} \subset
\QQ$ (endowed with the order induced by the natural order of $\QQ$).
The maximal ideal of this valuation domain is the ideal generated by
the elements $x^{1/p^k}$ with $k\in \N^\ast$, which is not principal.

\item Another example of the same type can be obtained by considering
the direct limit of all rings of the form $K[x^{1/n}]$ over all
non-zero natural numbers $n \in \N^\ast$.  The resulting valuation
domain has value group $\QQ$.  Therefore, it is not Noetherian. This
valuation domain has no non-zero prime element.
\end{enumerate}
\end{example}

\subsection{Discrete valuation domains}
\label{sec:dvr}

A {\em discrete valuation domain} is a Noetherian valuation domain
which is not a field, i.e. whose maximal ideal is non-zero. By
\cite[Chap. II.1, Exercise 1.4]{FS}, a valuation domain is Noetherian
iff its unique maximal ideal $m$ satisfies $\cap_{n\geq 1} m^n=0$.
Notice that a valuation domain with non-zero principal maximal ideal
need not be a discrete valuation domain (see the Example
\ref{ex:nonNoetheriandvr}).  The following characterizations are
well-known (see, for example, \cite[Proposition 6.3.4]{HS}):

\

\begin{Proposition}
Let $R$ be an integral domain which is not a field and let $K
\neq R$ be its field of fractions. Then the following statements are
equivalent:
\begin{enumerate}[(a)] \itemsep 0.0em
\item $R$ is a discrete valuation domain.
\item $R$ is a valuation domain with value group isomorphic to $\Z$
with its natural order.
\item Every prime ideal of $R$ is principal {\rm \cite[Chap. II.1,
Exercise 1.3]{FS}}.
\item $R$ is a principal ideal domain which has a unique non-zero
prime ideal.
\item $R$ is a principal ideal domain which has a unique prime element
$p$ up to association in divisibility.
\item $R$ is Noetherian and local and there is no ring $S$ such that
$R\subsetneq S \subsetneq K$.
\item $R$ is Noetherian of Krull dimension one and its maximal ideal
is principal.
\item $R$ is Noetherian of Krull dimension one and integrally closed.
\item $R$ is local with principal maximal ideal $m$ and we have
$\cap_{n\geq 1} m^n=0$.
\end{enumerate} In this case, the unique prime ideal of $R$ coincides
with the unique maximal ideal $m$ and we have $m=(p)$, where $p$ is
the essentially unique prime element (called {\em uniformizer}) of
$R$. Moreover, the discrete valuation $v:R\rightarrow \Z$ satisfies
$v(p)=1$ and any non-zero ideal of $R$ has the form $(p^n)$ for some
$n\geq 0$.
\end{Proposition}

\

\noindent In particular, any valuation domain which is not a field and
whose value group is not order-isomorphic to $\Z$ is
non-Noetherian. The following result (which follows immediately from
Proposition \ref{prop:cfgvr}) recovers a statement which, in this
Noetherian situation, also follows from the Buchweitz correspondence
\cite{Buchweitz}:

\

\begin{Proposition}
\label{prop:dvr} Let $R$ be a discrete valuation domain. Fix a
$\Z$-valuation $v:K\rightarrow \Z$ and a uniformizer $p$ of $R$. Then
any critically-finite element of $R$ has the form $W =up^n$, where
$n=v(W)\geq 2$ and $u$ is a unit of $R$. Given such an element of
$R$, the category $\hmf(R,W)$ is triangle-equivalent to
$\umod_{R/\langle p^n\rangle}$.
\end{Proposition}

\

\subsection{Constructions through the group of divisibility}
\label{sec:divgroup}

Recall that the {\em group of divisibility} $G(R)$ of an integral
domain $R$ is the quotient $K^\times/U(R)$, where $K$ is the
quotient field of $R$ and $U(R)$ is the group of units. It is an
ordered Abelian group when endowed with the order induced by the
divisibility relation. The group of divisibility of a B\'ezout domain
is lattice-ordered. In fact, any lattice-ordered
Abelian group $G$ is the group of divisibility of some B\'ezout domain
$R$ which can be obtained explicitly from $G$ by a construction due
to Jaffard and Ohm (see \cite{Jaffard,Ohm}). There exists a dictionary
between ideals of the B\'ezout domain $R$ associated to $G$ through
the construction given in op. cit. and the set of \emph{positive
filters} of $G$. Given a lattice-ordered Abelian group
$(G,\leq)$ and an element $x \in G$, the up and down sets determined
by $x$ are defined via $\uparrow x\eqdef \{y\in G\,|\,x\leq y\}$ and
$\downarrow x\eqdef \{y\in G\,|\,y\leq x\}$. A positive filter of
$(G,\leq)$ is defined to be a proper subset $F \subset G_+$ such that:
\begin{enumerate}[1.]  \itemsep 0.0em
\item $F$ is upward-closed, i.e. $x\in F$ implies $\uparrow x\subset
F$.
\item $F$ is closed under finite meets, i.e. $x,y\in F$ implies
$\inf(x,y)\in F$.
\end{enumerate}
\noindent A positive filter $F$ is called \emph{prime} if $G_+
\setminus F$ is a semigroup; it is called \emph{principal} if it has
the form $\uparrow x$ for some $x \in F$. The natural projection $\pi:
K^\times \to G$ induces a one to one correspondence between proper
ideals of $R$ and positive filters of $(G,\leq)$. Thus prime
ideals correspond to prime positive filters and non-zero principal
ideals correspond to principal positive filters. For more details and
precise statements we refer the reader to \cite[Section
5.2]{elbezout}.

It is an open question whether every B\'ezout domain is an elementary
divisor domain. Here we consider a class of lattice-ordered Abelian
groups which correspond to adequate B\'ezout domains (see Definition
\ref{def:adequate}), which are special cases of elementary divisor
domains (see \cite{Helmer2,Henriksen3} and Appendix \ref{app:edd}).

\

\begin{Definition} 
Let $(G,\leq)$ be a lattice-ordered Abelian group
and let $G^{+}=\{x \in G \,|\, x \geq 0\}$ denote its positive cone. We
say that $(G,\leq)$ is \emph{adequate} or \emph{projectable} 
if for every $a,b \in G^{+}$ there exist $r,s \in G^{+}$
satisfying the following conditions:
\begin{enumerate} \itemsep 0.0em
\item $a=r+s$.
\item $\inf(r, b)=0$.
\item If $t\in G$ satisfies $0 < t \leq s$, then we have $\inf(t, b)
\neq 0$.
\end{enumerate}
\end{Definition}

\

\noindent There exists a simple criterion for detecting adequate
groups.  Let $G$ be a lattice-ordered group. For any $b\in G^+$,
define $G_b^+=\{a \in G^+| \inf(a,b)=0\}$ and $G_b=\{a_1-a_2\, |\, a_1,a_2
\in G_b^+\}$. It is easy to see that $G_b$ is a lattice subgroup of
$G$. Then \cite[Theorem 4.7]{LLS} states that $(G,\leq)$ is adequate
iff $G_b$ is a summand of $G$ for every element $b\in G^+$.

\

\begin{Proposition}{\rm \cite{LLS}}
\label{prop:lls} 
Let $(G,\leq)$ be an adequate lattice-ordered Abelian group. Then:
\begin{enumerate}
\item The B\'ezout domain $R$ associated to $(G,\leq)$ by the
Jaffard-Ohm construction is an adequate B\'ezout domain (and hence
also an elementary divisor domain).
\item The prime elements of $R$ correspond to the principal prime
positive filters of $(G,\leq)$.
\end{enumerate}
\end{Proposition}

\begin{proof} The fact that $G$ is adequate was shown in
\cite{LLS}. On the other hand, any adequate B\'ezout domain is an
elementary divisor domain (see \cite{Helmer2,Henriksen3}). The second
statement follows immediately from the discussion above.  \qed
\end{proof}

If $R$ is a B\'ezout domain with prime elements which is constructed
from an adequate lattice-ordered group as in Proposition
\ref{prop:lls} and $W \in R$ is a critically-finite element, then
Theorem \ref{thm:main} applies to the homotopy category of finite rank
matrix factorizations of $W$ over $R$.

\begin{example} Let $I$ be a non-empty set and let $G$ be either the
direct sum or the direct product of a family of totally ordered groups
$(G_i)_{i\in I}$ indexed by $I$. Then the B\'ezout domain $R$
associated to $G$ is adequate (see \cite[Corollary 4.8]{LLS}). 
The prime elements of the corresponding
elementary divisor domain $R$ were described in \cite[Section
5.2]{elbezout}. Let $W$ be a critically finite element of $R$. By
Proposition \ref{prop:hmfdec} the category $\hmf(R,W)$ has an
orthogonal decomposition indexed by the critical prime divisors of
$W$.
\end{example}

\subsection{Constructions through spectral posets}
\label{subsec:sposet}

The {\em spectral poset} of a unital commutative ring $R$ is the prime
spectrum $\Spec(R)$ endowed with the order relation $\leq$ given by
inclusion.  For two elements $x,y$ of a poset $(X,\leq)$, we write
$x\ll y$ if $x<y$ and $x$ is an immediate neighbor of $y$. The
spectral poset of any unital commutative ring satisfies {\em
Kaplansky's conditions} (see \cite{KaplanskyBook}):
\begin{enumerate}[I.]  
\itemsep 0.0em
\item Every non-empty totally-ordered subset of $(\Spec(R),\leq)$ has
a supremum and an infimum (in particular, $\leq$ is a lattice order).
\item Given any elements $x,y\in \Spec(R)$ such that $x<y$, there
exist distinct elements $x_1,y_1$ of $\Spec(R)$ such that $x\leq
x_1<y_1\leq y$ and such that $x_1\ll y_1$.
\end{enumerate}
A poset $(X,\leq)$ is called a {\em tree} if for
every $x\in X$, the lower set $\downarrow x=\{y\in X|y \leq x\}$ is
totally ordered. One has the following result due to Lewis:

\

\begin{Theorem}{\rm \cite{Lewis}}
\label{thm:Lewis}
Let $(X,\leq)$ be a partially-ordered set. Then the following
statements are equivalent:
\begin{enumerate}[(a)]
\itemsep 0.0em
\item $(X,\leq)$ is a tree which has a unique minimal element $\theta\in X$ and
  satisfies Kaplansky's conditions I. and II.
\item $(X,\leq)$ is isomorphic with the spectral poset of a B\'ezout
  domain.
\end{enumerate}
Moreover, $R$ is a valuation domain iff $(X,\leq)$ is a totally-ordered set. 
\end{Theorem}

\

\noindent The B\'ezout domain in Theorem \ref{thm:Lewis} is obtained by
associating a lattice-ordered group $G$ to the poset $(X,\leq)$ and
applying the Jaffard-Ohm construction to $G$. The following
result was proved in \cite{elbezout}:

\

\begin{Proposition}
\label{prop:spec}
Let $(X,\leq)$ be a tree which has a unique minimal element and
satisfies Kaplansky's conditions I. and II. and let $R$ be the
B\'ezout domain determined by $(X,\leq)$ as explained above. Then
for each maximal element $x$ of $X$ which belongs to the set 
\be
X^\ast\eqdef \big\{x\in X\,|\,\exists y\in X: y\ll x\big\}~~,
\ee
the principal positive filter $\uparrow 1_x$ is prime and hence
corresponds to a principal prime ideal of $R$. Moreover, we have:
\ben
\label{up1x}
\uparrow 1_x=\big\{f\in G_+ \, | \, \supp(f)\cap \downarrow x\neq \emptyset\big\}
\een
and:
\ben
\label{Fx1x}
F_x=\big\{f\in \, \uparrow 1_x \, | \, \inf S_f(x) \in S_f(x)\big\}=\big\{f\in \, \uparrow 1_x \, | \, \exists \min S_f(x)\big\}~~,
\een
where: 
\be
S_f(x)\eqdef \supp(f)\cap \downarrow x~~.
\ee
\end{Proposition}

\

\noindent A particularly simple example of elementary divisor domains
is provided by those B\'ezout domains $R$ which are {\em $PM^\ast$
rings}, i.e. which have the property that any non-zero prime ideal of
$R$ is contained in a unique maximal ideal (see Theorem
\ref{thm:PMast} in Appendix \ref{app:edd}).

\

\begin{Definition}
\label{def:pm} A tree $(X,\leq)$ is called a {\em $PM^\ast$ tree} if
the following three conditions hold:
\begin{enumerate} \itemsep 0.0em
\item $X$ has a unique minimal vertex $\theta$ (called the {\em root}).
\item $X$ satisfies Kaplansky's conditions I. and II.
\item $X$ is branched only at the root, i.e. for every $x\in
X\setminus \{\theta\}$, there exists at most one element $y\in X$ such
that $x\ll y$.
\end{enumerate}
\end{Definition}

\

\begin{Proposition}
\label{prop:tree} Let $X$ be a $PM^\ast$ tree and $R$ be the B\'ezout
domain associated to $X$ as explained above. Then $R$ is a $PM^\ast$
ring and hence an elementary divisor domain.
\end{Proposition}

\

\begin{proof} Condition 3. in definition \ref{def:pm} implies that
every element $x\in X\setminus \{\theta\}$ is bounded from above by a
unique maximal element of $X$. Since the elements of $X\setminus
\{\theta\}$ correspond to the non-zero prime ideals of $R$, this
implies that any non-zero prime ideal of $R$ is contained in a unique
maximal ideal. Thus $R$ is a $PM^\ast$ ring. Since $R$ is also a
B\'ezout domain by Theorem \ref{thm:Lewis}, we conclude by Theorem
\ref{thm:PMast} that $R$ is an elementary divisor domain.\qed
\end{proof}

\

\begin{example}

\

\begin{enumerate}
\item Let $X$ be a tree with a unique minimal element which satisfies
Kaplansky's conditions I. and II.  Assume that the set of maximal
vertices of $X$ is countable. Then it was shown in \cite{K2} that the
associated B\'ezout domain $R$ is an elementary divisor domain.  As a
simple example, consider a countable corolla $T$ as in \cite[Example
5.8]{elbezout}. The vertices of $T$ are the elements of the set
$\N=\Z_{\geq 0}$, with the partial order given by $0<x$ for every $x
\in \N^\ast=\Z_{>0}$ and no further strict inequality. The root of $T$
is the element $0\in \N$ while every maximal vertex $x \in \N^\ast$
corresponds to a principal prime ideal of the associated B\'ezout
domain.
\item If we replace each edge of the countable corolla $T$ discussed
above with some finite tree, then the collection of maximal vertices
of the resulting tree $T'$ is still countable and the associated
B\'ezout domain $R'$ is an elementary divisor domain which need not be
a $PM^\ast$ ring.
\end{enumerate}
\end{example}

\appendix

\section{Matrices over a GCD domain}
\label{app:GCD}

Recall that an integral domain $R$ is called a {\em GCD domain} if any two elements
$f,g\in R$ admit a greatest common divisor (gcd). In this case, 
any non-empty finite collection of elements $f_1,\ldots, f_n\in R$
admits a gcd and and lcm, both of which are determined up to
association and whose classes we denote by:
\be
(f_1,\ldots, f_n)\in R/U(R)~~\mathrm{and}~~[f_1,\ldots, f_n]\in R/U(R)~~.
\ee
The gcd class $(f)$ of a single element $f\in R$ coincides with the
equivalence class of $f$ under association in divisibility.

\

\begin{Definition}
Let $A\in \Mat(m,n,R)$ be an $m$ by $n$ matrix with coefficients from a GCD domain
$R$. For any $k\in \{1,\ldots, r\}$, the {\em $k$-th determinantal invariant}
$\bdelta_k(A)\in R/U(R)$ of $A$ is defined to be the gcd class of all $k\times k$
minors of $A$. We also define $\bdelta_0(A)=(1)$.
\end{Definition}

\

\begin{Proposition} {\rm \cite{Friedland}}
Let $R$ be a GCD domain. For any $A\in \Mat(m,n,R)$, we have: 
\be
\bdelta_{k-1}(A)|\bdelta_{k}(A)~,~~\forall k\in \{1,\ldots, \rk A\}~~.
\ee
Defining the {\em invariant factors} $d_k(A)\in R/U(R)$ by:
\be
\bd_k(A)\eqdef \twopartdef{\frac{\bdelta_k(A)}
{\bdelta_{k-1}(A)}}{\bdelta_{k-1}(A)\neq 0}{\;\;\;(1)}{\bdelta_{k-1}(A)=0}~,~~\forall k\in \{1,\ldots, \rk A\}~~,
\ee
we have: 
\be
\bd_{k-1}(A) | \bd_k(A)~,~~\forall k \in \{2,\ldots, \rk A\}~~.
\ee
\end{Proposition}

\

\begin{Proposition} {\rm \cite{Friedland}}
Let $R$ be a GCD domain and $A,B\in \Mat(m,n,R)$. If $A$ and $B$ are equivalent, then 
$\rk(A)=\rk(B)=r$ and $\bd_k(A)=\bd_k(B)$ for all $k\in \{1,\ldots,r\}$. 
\end{Proposition}

\section{Elementary divisor domains}
\label{app:edd}

\noindent In this appendix, we collect some facts about elementary divisor domains. 

\

\begin{Definition}
An integral domain $R$ is called an \emph{elementary divisor domain}
(EDD) if for any three elements $a,b,c\in R$, there exist $p,q,x,y\in
R$ such that $(a,b,c)=pxa+pyb+qyc$ is a GCD of $a$, $b$ and $c$.
\end{Definition}

\subsection{Examples of elementary divisor domains}

The following are examples of elementary divisor domains:
\begin{itemize} \itemsep 0.0em
\item Any B\'ezout domain which is an $F$-domain (i.e. for which any
non-zero element is contained in at most a finite number of maximal
ideals) is an EDD \cite[Sec. 4]{ButtsDulin}.  In particular, any PID
is an EDD.
\item The ring $\A$ of algebraic integers is an EDD \cite[Theorem
5]{NT} which has no prime elements.
\item The ring of entire functions defined on the complex plane is an
EDD \cite{Helmer2,Helmer}. The prime elements of this ring are the
entire functions which have a single simple zero in the complex plane.
\item If $R$ is an EDD with quotient field $K$ and $J$ is any integral
domain such that $R\subset J\subset K$, then $J$ is an EDD
\cite[Sec. 4]{ButtsDulin}.  When $R$ is a PID, it is known that any
domain $J$ of this type is a PID and hence Noetherian.
\item Any Kronecker function ring is an EDD \cite{EO}.
\item Any generalized valuation domains is an EDD. If $V_1,\ldots,
V_n$ are generalized valuation domains with the same quotient field
$K$, then $R\eqdef \cap_{i=1}^n V_i$ is an EDD \cite[Sec. 4]
{ButtsDulin}.
\item The domains formed by Jaffard's pull-back theorems are EDDs
\cite[Sec. 4]{ButtsDulin}.
\item Let $B$ be an EDD with quotient field $K$ and let $m$ be the
maximal ideal of the power series ring $K[[x]]$ in one variable. Then
$R:=B+m$ is an EDD \cite[Sec. 4]{ButtsDulin}.
\item Let $B$ be an EDD with quotient field $K$ and $X$ be an
indeterminate. Then $R:=B+XK[X]$ is an EDD \cite{CMZ}.
\item Let $K$ be an algebraically closed field of characteristic
different from two and let $x_1$ be an indeterminate over $K$. Let
$x_2$ be a square root of $x_1$, $x_3$ be a square root of $x_2$ and
so on.  Then the ring $R:=\cup_{n=1}^\infty K[x_n,1/x_n]$ is an EDD
\cite[Sec. 4]{ButtsDulin}.
\end{itemize}

\subsection{Kaplansky's characterization of EDDs}

\

\

\begin{Definition}
Let $R$ be a commutative ring. We say that $R$ satisfies
\emph{Kaplansky's condition} if for any three elements $a,b,c$ in $R$
such that $(a,b,c)=(1)$, there exist elements $p,q \in R$ such
that $(pa, pb+qc)=(1)$.
\end{Definition}

\

\begin{Proposition} {\rm \cite{K}}
\label{K}
An integral domain $R$ is an EDD iff it satisfies the following two conditions:
\begin{enumerate}
\itemsep 0.0em
\item $R$ is a B\'ezout domain.
\item $R$ satisfies Kaplansky's condition.
\end{enumerate}
\end{Proposition}

\subsection{The Smith normal form theorem over an EDD}

\

\

\begin{Theorem} {\rm \cite{Friedland}}
Let $R$ be an EDD. For any matrix $A\in \Mat(m,n,R)$, there exist
matrices $U\in \GL(m,R)$ and $V\in \GL(n,R)$ such that:
\be
UAV^{-1}=D~,
\ee
where $D_{ij}=0$ for all $i\neq j$ and the diagonal entries $d_i\eqdef
D_{ii}$ (with $i\in \{1,\ldots, r\}$, where $r\eqdef \rk A \leq \min(m,n)$) 
are non-zero elements which satisfy the condition:
\be
d_1|d_2|\ldots |d_r~~.
\ee
In this case, the matrix $D$ is called the {\em Smith normal form} of
$A$. Moreover, the association classes of $d_k$ coincide with the invariant factors of $A$:
\be
(d_k)=\bd_k(A)~~,~~\forall k\in \{1,\ldots, r\}~~.
\ee
\end{Theorem}

\

\begin{Proposition} {\rm \cite{Friedland}}
\label{prop:equivEDD}
Let $R$ be an EDD and $A,B\in \Mat(m,n,R)$. Then $A$ and $B$ are
equivalent iff they have the same rank $r$ and their invariant factors coincide:
\be
\bd_k(A)= \bd_k(B)~,~~\forall k\in \{1,\ldots, r\}~~.
\ee
\end{Proposition}

\subsection{Some special classes of EDDs}

It is an unsolved problem (going back at least to \cite{Helmer2})
whether any B\'ezout domain is an EDD. Here we mention a few special
classes of B\'ezout domains which are known to be elementary divisor
domains. One special class is provided by those B\'ezout domains which
are $PM^\ast$-rings.

\

\begin{Definition}{\rm \cite{McGovern}}
\label{def:PMast}
A {\em $PM^\ast$}-ring is a unital commutative ring $R$ which has the
property that any non-zero prime ideal of $R$ is contained in a {\em
  unique} maximal ideal of $R$.
\end{Definition}

\

\begin{Theorem}{\rm \cite{Zabavsky}}
\label{thm:PMast}
Let $R$ be a B\'ezout domain which is a $PM^\ast$ ring. Then $R$ is an
EDD.
\end{Theorem}

\

\noindent It was shown in \cite{Zabavsky2} that a B\'ezout domain is an
EDD iff it has Gelfand range one.

\

\noindent Another special class is that of adequate B\'ezout domains \cite{LLS,Helmer2,GH}. 

\

\begin{Definition}{\rm \cite{Helmer2}}
\label{def:adequate}
A B{\'e}zout domain $R$ is called {\em adequate} if for all $a,b\in R$
with $a\neq 0$, there exist $r,s\in R$ such that $a=rs$, $(r,b)=R$ and
such that any non-unit $s'$ which divides $s$ satisfies $(s',b)\neq
R$.
\end{Definition}

\

\begin{Proposition}{\rm \cite{Henriksen3}}
Any adequate B{\'e}zout domain is a $PM^\ast$ ring.
\end{Proposition}

\

\begin{Corollary}{\rm \cite{Helmer2}}
Any adequate B{\'e}zout domain is an EDD.
\end{Corollary}

\

\begin{remark} It is known that the inclusions: \be
\{\mathrm{adequate~rings}\}\subset \{PM^\ast~\mathrm{rings}\}\subset
\{\mathrm{elementary~divisor~domains}\} \ee are strict (see
\cite{Zabavsky,BCM}).
\end{remark}

\

\begin{Theorem} The ring $\O(\Sigma)$ of entire functions on any
connected and non-compact borderless Riemann surface is an adequate
B\'ezout domain.
\end{Theorem}

\

\noindent The case $\Sigma=\C$ of this theorem was established in
\cite{Helmer2,Henriksen3}. This generalizes to any Riemann surface
using \cite{Alling1,Alling2}. Since $\O(\Sigma)$ is an adequate B\'ezout
domain, it is also a $PM^\ast$ ring and hence and EDD.

\subsection{The Noetherian case}
The following characterizations are well-known.

\

\begin{Proposition}
Let $A$ be a Noetherian integral domain. 
Then the following statements are equivalent: 
\begin{enumerate}
\itemsep 0.0em
\item $A$ is an EDD.
\item $A$ is a B\'ezout domain.
\item $A$ is a PID.
\end{enumerate}
\end{Proposition}

\

\noindent In particular, matrices valued in a Noetherian
domain $A$ admit a Smith normal form iff $A$ is a PID. It is obvious
that every PID is Noetherian.

\

\begin{Proposition}
Let $A$ be an integral domain. Then the following statements are
equivalent:
\begin{enumerate}
\itemsep 0.0em
\item $A$ is a PID.
\item $A$ is a UFD and a B\'ezout domain.
\item $A$ is a UFD and a Dedekind domain.
\item $A$ is a UFD and has Krull dimension one (equivalently, any
  non-zero prime ideal is maximal).
\end{enumerate}
\end{Proposition}

\

\begin{Proposition}
Let $A$ be a Noetherian integral domain. Then the following statements
are equivalent:
\begin{enumerate}
\itemsep 0.0em
\item $A$ is a UFD.
\item $A$ is normal and its divisor class group vanishes.
\item Every height one principal ideal of $A$ is principal.
\end{enumerate}
\end{Proposition}

\

\begin{acknowledgements}
This work was supported by the research grant IBS-R003-S1. 
\end{acknowledgements}

\

\end{document}

%% file: modquiver.pdf_t
\begin{picture}(0,0)%
\includegraphics{modquiver.pdf}%
\end{picture}%
\setlength{\unitlength}{4144sp}%
\begingroup\makeatletter\ifx\SetFigFont\undefined%
\gdef\SetFigFont#1#2#3#4#5{%
  \reset@font\fontsize{#1}{#2pt}%
  \fontfamily{#3}\fontseries{#4}\fontshape{#5}%
  \selectfont}%
\fi\endgroup%
\begin{picture}(5794,925)(1111,-4013)
\put(2431,-3886){\makebox(0,0)[lb]{\smash{{\SetFigFont{14}{16.8}{\rmdefault}{\bfdefault}{\updefault}{\color[rgb]{0,0,0}$V_2$}%
}}}}
\put(3871,-3886){\makebox(0,0)[lb]{\smash{{\SetFigFont{14}{16.8}{\rmdefault}{\bfdefault}{\updefault}{\color[rgb]{0,0,0}$V_3$}%
}}}}
\put(5311,-3796){\makebox(0,0)[lb]{\smash{{\SetFigFont{14}{16.8}{\rmdefault}{\bfdefault}{\updefault}{\color[rgb]{0,0,0}$V_4$}%
}}}}
\put(6661,-3796){\makebox(0,0)[lb]{\smash{{\SetFigFont{14}{16.8}{\rmdefault}{\bfdefault}{\updefault}{\color[rgb]{0,0,0}$V_5$}%
}}}}
\put(1126,-3931){\makebox(0,0)[lb]{\smash{{\SetFigFont{14}{16.8}{\rmdefault}{\bfdefault}{\updefault}{\color[rgb]{0,0,0}$V_1$}%
}}}}
\end{picture}%

%% file: squiver.pdf_t
\begin{picture}(0,0)%
\includegraphics{squiver.pdf}%
\end{picture}%
\setlength{\unitlength}{4144sp}%
\begingroup\makeatletter\ifx\SetFigFont\undefined%
\gdef\SetFigFont#1#2#3#4#5{%
  \reset@font\fontsize{#1}{#2pt}%
  \fontfamily{#3}\fontseries{#4}\fontshape{#5}%
  \selectfont}%
\fi\endgroup%
\begin{picture}(4489,925)(1111,-4013)
\put(2431,-3886){\makebox(0,0)[lb]{\smash{{\SetFigFont{14}{16.8}{\rmdefault}{\bfdefault}{\updefault}{\color[rgb]{0,0,0}$V_2$}%
}}}}
\put(3871,-3886){\makebox(0,0)[lb]{\smash{{\SetFigFont{14}{16.8}{\rmdefault}{\bfdefault}{\updefault}{\color[rgb]{0,0,0}$V_3$}%
}}}}
\put(5311,-3796){\makebox(0,0)[lb]{\smash{{\SetFigFont{14}{16.8}{\rmdefault}{\bfdefault}{\updefault}{\color[rgb]{0,0,0}$V_4$}%
}}}}
\put(1126,-3931){\makebox(0,0)[lb]{\smash{{\SetFigFont{14}{16.8}{\rmdefault}{\bfdefault}{\updefault}{\color[rgb]{0,0,0}$V_1$}%
}}}}
\end{picture}%